\documentclass[11pt]{article}
\usepackage{epsfig,amsmath,amssymb,latexsym,color}
\usepackage{listings}
\usepackage{color}
\usepackage{verbatim}
\usepackage{tikz}
\usepackage{cite}
\usepackage{setspace}
\usepackage[]{algorithm2e}

\usetikzlibrary{shapes,snakes,arrows}

\newcounter{mnote}
  \let\oldmarginpar\marginpar
 \renewcommand\marginpar[1]{\-\oldmarginpar[\raggedleft\footnotesize #1]%
    {\raggedright\footnotesize #1}}

\def\bfx{{\bf x}}

\newcommand{\abs}[1]{\lvert#1\rvert}
\newcommand\norm[1]{\left\lVert#1\right\rVert}

\newcommand{\pp}{X_}
\newcommand{\qq}{Y_}

\newtheorem{theorem}{Theorem}
\newtheorem{lemma}{Lemma}
\newtheorem{proposition}{Proposition}
\newtheorem{corollary}{Corollary}

\newtheorem{example}{Example}
\newtheorem{definition}{Definition}
\newenvironment{proof}{\begin{trivlist}\item[]{\emph{Proof.}}}
               {\hfill$\Box$\end{trivlist}}

\begin{document}
\title{On Convergence of the Alternating Projection Method \\ 
for  Matrix Completion and Sparse Recovery Problems}

\author{
Ming-Jun Lai 
\footnote
{Department of Mathematics,
University of Georgia, Athens, GA 30602,
mjlai@uga.edu. 
This research is partially supported 
the National Science Foundation under the grant \#DMS 1521537. }
\and Abraham Varghese 
\footnote
{Department of Mathematics, University of Georgia, Athens, GA 30602, opticsabru@gmail.com.}}

\maketitle

\section{Introduction}
The last two decades have witnessed a resurgence of research in sparse 
solutions of underdetermined linear systems and matrix completion 
and recovery. 
The matrix completion problem was inspired by Netflix problem 
(cf. \cite{netflix}) and 
was pioneered by [Cand\`es, Recht, 2010\cite{CR09}] and [Cand\`es and Tao, 2010\cite{CT10}]. 
The problem can be explained as follows. 
One would like to recovery  a matrix $M\in \mathbb{R}^{m,n}$ 
from a given set of  entries $M_{ij}, (i,j)\in \Omega
\subset \{1, \cdots, m\}\times \{1, \cdots, n\}$ by filling in the 
missing entries such that the resulting matrix has the lowest possible rank. 
In other words, we solve the following rank minimization problem:
\begin{equation}
\label{rm}
\min_{X\in \mathbb{R}^{m\times n}} \quad \hbox{rank } (X): \quad 
\hbox{ such that }
\quad \mathcal{A}_\Omega(X)=\mathcal{A}_\Omega(M),
 \end{equation}
 where $\mathcal{A}_\Omega(X)=\mathcal{A}_\Omega(M)$ means the entries 
of  the matrix $X$  are  the same entries of matrix $M$ for  indices 
$(i,j)\in \Omega$.  
Clearly, if we are only given a few entries, say one entry of matrix $M$
 of size $2\times 2$, we are not able to recover $M$ even assuming the  rank of $M$ is $1$. There are 
necessary conditions on how many entries one must know in order to be able to recover $M$. 
Information theoretic lower bound can be found in \cite{CT10}.
 
There are many approaches to recovery such a matrix developed in the last ten years. 
One popular approach is to find a matrix with minimal summation of its singular values. That is,  
\begin{equation}
\label{mainproblemMC}
\min_{X\in \mathbb{R}^{m\times n}} \{\|X\|_*,  \qquad {\cal A}_\Omega (X) = {\cal A}_\Omega(M)\}, 
\end{equation}
where $\|X\|_*= \sum_{i=1}^k \sigma_i(X)$ is the nuclear norm of $X$ 
with $k=\min\{m, n\}$ and $\sigma_i(X)$ are singular values of matrix 
$X$. It is known that $f(X)=\|X\|_*$ is a convex function of $X$, 
the above problem (\ref{mainproblemMC}) is a convex minimization 
problem. By adding $\dfrac{1}{\lambda}\|X\|_F$ to the minimizing functional 
in (\ref{mainproblemMC}), the resulting minimization problem can 
be solved by using Uzawa type algorithms in  
[Cai, Cand\`es, Shen, 2010\cite{CCS10}] or solved by using its dual 
formulation, e.g. in [Lai and Yin, 2013\cite{LY13}. The minimization
in (\ref{mainproblemMC}) can also reformulated as a fixed point 
iteration and Nestrov's acceleration technique can be used. See 
[Ma, Goldfarb, Chen, 2011\cite{MGC11}] and [Toh and Yun, 2010\cite{TY10}].  
This constrained minimization (\ref{mainproblemMC}) 
is usually converted into an  unconstrained minimization using Lagrange
multiplier method or augmented Lagrange minimization method. 
The alternating direction method of multiplier (ADMM) can be used to complete a matrix. 
See [Tao and Yuan, 2011\cite{TY11}],  and [Yang and Yuan, 2012\cite{YY12}],
Many researchers have studied the matrix completion via variants of the constrained 
convex minimization approach. 

Certainly, the rank completion is also studied by using other approaches. 
See  [Jain, Meka and Dhillon, 2010\cite{JMD10}] for singular value projection method and  
[Wen, Yin, Zhang, 2012\cite{WYZ12}, 
[Tanner and Wei, 2016\cite{TW16}] for alternating least squares, the SOR approaches, 
steepest descent minimization approaches. See [Lai, Xu, Yin, 2013\cite{LXY13}] for $\ell_q$ minimization 
approach for $q\in (0, 1)$. In addition,  a greedy approach, e.g. orthogonal matching pursuit 
(OMP) and iterative hard thresholding approach  can be used as well. See [Wang, Lai, Lu, and Ye, 2015 
\cite{WL15}] and [Tanner and Wei, 2013\cite{TW13}].  Iteratively reweighted nuclean norm minimization, 
Riemannian conjugated gradient method, and alternating projection algorithm 
in  [Mohan and Fazel, 2012\cite{MF12}], 
[Vandereycken, 2013\cite{V13}],   [Cai, Wang, and Wei, 2016\cite{CWW16}], 
[Wei, Cai, Chan, and Leung, 2016\cite{WCCL16}],  [Jiang, Zhong, Liu, and Song, 2017\cite{JZLS17}, and etc..  
Among all these algorithms, the computational algorithm proposed in \cite{WL15} seems the most 
efficient one in completing an incomplete matrix. However, the accuracy of the completed matrices 
is still a question.  Usually, the researchers use the relative Frobenius norm errors,
i.e., $\|M- M_k\|_F/\|M\|_F$ to measure the accuracy for a given matrix $M$ of size $m\times n$, 
where $M_k$ is the $k$th iteration from a matrix completion algorithm. When the size of $M$ is very large 
and so is $\|M\|_F$ and the missing rate $1-|\Omega|/(mn)<<1$ small or $|\Omega| \approx mn$, 
the relative Frobenous norm error will be very small anyway and hence will not be a good measure of errors, 
where $\Omega$ is the set of the indices of known entries. 
It is better to use a true error such as the maximum norm of all entries of the residual matrix 
$M-M_k$ to check the accuracy of the completed matrices. Then many of 
the existing algorithms mentioned above will fail to produce a good recovery. Certainly, the main possible reason may be that the relaxation of rank minimization problem is used.
One of the mathematical problems is to find sufficient conditions 
which ensure the uniqueness of the minimization. However, some sufficient conditions are unrealistic, 
e.g. only one entry is missing. Another research problem is to design efficient matrix completion algorithms. 
It is interesting to have an algorithm which is convergent fast, say in a linear fashion.    

Recently, the authors just discovered the reference \cite{JZLS17} which presents a numerical study of a
computational algorithm called Alternating Projection(AP) Algorithm which is  exactly the algorithm 
that the authors of this paper have studied for a year. Mainly, we also had the same algorithm and  
was aware of the good numerical performance. However, we would like to know why the algorithm is convergent, 
under what kind of conditions the algorithm is convergent, and under what situation the convergence is linear. 
The study took many months and delayed our announcement of the AP algorithm. Nevetheless, 
the purpose of this paper is to explain why and when the AP Algorithm  will converge and the convergence 
is linear. In addition, we shall explain the existence of matrix completion and how many matrices can be
completed to have the same given entries.  
In general, for randomly chosen values for a fixed location set $\Omega$ to 
be known entries of a matrix, one will not be able to complete it by using a rank $r$ matrix.  
Hence, we shall discuss the convergence of the AP algorithm under the assumption 
that the given entries are from a matrix with rank $r$.  Also, we will  
provide an approach to choose a good initial guess such that 
the AP Algorithm will converge faster than using the simple straightforward initial guess $M_\Omega$ as used  
in \cite{JZLS17}. An application to image process will be shown to demonstrate a nice performance 
of the AP algorithm.  
Finally, we shall extend the ideas from the AP Algorithm to find sparse solution 
of under determined linear systems. 

Indeed, the matrix completion problem  is closely related to the sparse vector recovery problem. 
Sparse solutions of underdetermined linear systems have been studied for last twenty years starting 
from [Chen, Donoho, Saunders, 1998\cite{CDS98}] and [Tibshirani, 1996\cite{T96}] 
and then  became a major subject of research as a part of compressive sensing study since 
2006 due to [Donoho, 2006\cite{D06}], [Cand\'es, 2006\cite{C06}], [Cand\'es and Tao, 2005\cite{CT05}], and  
[Cand\'es, Romberg, and Tao, 2006\cite{CRT06}]. 
Many numerical algorithms  have been developed since then. Several algorithms are based on classic 
convex minimization approach (cf. e.g. \cite{HYZ08}, \cite{BT09}, \cite{LY13}, and etc.). 
Several algorithms are based on iteratively reweighted $\ell_1$ minimization or $\ell_2$ 
minimizations (cf. [Cand\'es, Watkin, and Boyd, 2008\cite{CWB08}], [Daubechies, DeVore, 2010,\cite{DDFG10}] 
and [Lai, Xu, and Yin, 2013\cite{LXY13}]). Several researchers started the $\ell_q$ minimization 
for $q\in (0, 1)$, e.g. in [Foucart and Lai, 2009\cite{FL09}] and [Lai and Wang, 2011\cite{LW11}].  
Various other algorithms are based on greedy or orthogonal matching pursuit (cf. e.g. 
[DeVore and Temlyakov, 1996\cite{DT96}], [Tropp, 2004\cite{T04}], and [Kozlov and Petukhov, 2010\cite{KP10}]).  
some algorithms are also based on the hard thresholding technique such as in 
[Blumensath and Davies, 2009\cite{BD09}], [Blumensath and Davies, 2010\cite{BD10}], 
[Foucart, 2011\cite{F11}] and etc.. Among the various other numerical methods were also proposed. 
See, e.g. [Dohono, Maleki, and Montanari, 2009\cite{DMM09}], 
[Rangan, 2011\cite{R11}], [Gong, Zhang, Lu, Huang, and Ye,2013\cite{GZLHY13}], 
[Wang and Ye, 2014\cite{WY13}] and etc.. 
To the best of our knowledge, the method in \cite{KP10} is the most effective in finding  sparse solutions. 
Thus, we shall extend the alternating projection method to the sparse recovery problem and establish 
some sufficient conditions that our algorithm is convergent and its convergence is linear.  

The paper is simply organized as follows. In the next section, we study the convergence of the AP algorithm.  
The section is divided into three subsections.  We first study the case that the guess rank $r_g$ is 
the same as the rank of the matrix to be completed. 
Next we study the remaining case that $r_g$ is not the same as the rank of the matrix whose known entries
are given. Finally in this section,  we show 
the excellent performance of the AP algorithm when starting from 
an initial matrix obtained from the OR1MP algorithm in \cite{WL15}. 
In \S 3, we extend the AP algorithm to the compressive sensing setting. \S 3 is divided into 
two subsections. First we study the convergence of 
the alternating projection algorithm for compressive sensing. Then we present some numerical experiments. 
Comparing with many known algorithms, the alternating projection method performs very well. 
Finally in this paper, we present an algebraic geometry analysis to show the existence of matrix 
completion and the number of matrices which can be completed from the given known entries of a rank $r$
matrix.

\section{The Alternating Projection Algorithm for Matrix Completion}
Let $\mathcal{M}_r$ be  the manifold in $\mathbb{R}^{n^2}$ consisting of $n\times n$  matrices
(without loss of generality) of rank $r$ 
and denote by $P_{\mathcal{M}_r}$ the projection operator onto the manifold $\mathcal{M}_r$. 
Next consider the affine space $\mathcal{A}_\Omega$ defined as follows:
$$
\mathcal{A}_\Omega := \left\{X \mid \mathcal{P}_\Omega(X-M) = 0 \right\}.
$$
Affine spaces $\mathcal{A}_\Omega$ consists of matrices which has exactly same entries as $M$ 
with indices in $\Omega$. Although it is a convex set, $\mathcal{A}_\Omega$ is not a bounded set. 
Starting with an initial guess $X_0 = \mathcal{P}_\Omega(M)$ or a good initial guess (see our numerical 
experiments near the end of this section), 
the Alternating Projection (AP) Algorithm can be simply stated as follows: 

\RestyleAlgo{boxruled}
\LinesNumbered
\begin{algorithm}[ht]
\label{alg1}
	\LinesNotNumbered
	\caption{Alternating Projection Algorithm for Matrix Completion}
	\KwData{Rank $r$ of the solution $M$, the tolerance $\epsilon$ whose default value is 1e-6}
	\KwResult{$X_k$, a close approximation of $M$}
	Initialize $X_0 = \mathcal{P}_\Omega(M)$ or any other good guess\;
	\Repeat{$\norm{X_{k+1}-X_k} < \epsilon$}{
		\textbf{Step 1:} $Y_k = P_{\mathcal{M}_r}(X_k)$\\
		\textbf{Step 2:} $X_{k+1} = P_{\mathcal{A}_\Omega}(Y_k)$ 
	}
\end{algorithm}

\noindent
In Algorithm~\ref{alg1} above, 
the computation of the projection $P_{\mathcal{M}_r}$ can be realized easily by using the singular 
value decomposition.   $P_{\mathcal{A}_{\Omega}}$ is the projection onto $\mathcal{A}_\Omega$.
The computation $P_{\mathcal{A}_{\Omega}}(Y_k)$ is obtained simply by setting the matrix 
entries of $Y_k$ in positions $\Omega$ equal to the corresponding entries in $M$. Therefore, this 
algorithm is simple and easy without any minimization. The algorithm is the same as one in [Jiang, Zhong, 
Liu, and Song, 2017\cite{JZLS17}]. One of the purposes of our paper is to show the convergence under various 
conditions. 

Before studying the convergence of Algorithm~\ref{alg1}, let us comment on the existence of a rank r matrix
which has the known entries in position $\Omega$. Let $m=|\Omega|$ be the cardinality of $\Omega$. We shall 
assume $m>2nr-r^2$. For convenience, we shall use the complex $m$ dimensional space $\mathbb{C}^m$ to
discuss the existence. We will show that if one randomly chooses the entries of a matrix $M$ in the positions
in $\Omega$ from $\mathbb{C}^m$, the probability of completing the matrix $M$ of rank $r$ is zero. See
Theorem~\ref{nochance}.  Thus, 
we have to assume that the given entries are from a rank r matrix $M$. In other words,  
we call a vector ${\bf x} \in \mathbb{C}^m$  $r$-\emph{feasible} if there exist a rank $r$ matrix $M$ 
such that $M|_\Omega = {\bf x}$. If the entries ${\bf x}\in \mathbb{C}^m$ over $\Omega$ 
are  $r$-feasible, we would like to know if there is a unique rank-r matrix $M$ satisfying $M|_\Omega
={\bf x}$. We can show that
number of ways to complete a matrix of rank $r$ is less than or equal to $\displaystyle 
\prod_{i=0}^{n-r-1}\frac{\binom{n+i}{r}}{\binom{r+i}{r}}$ in general. See Theorem~\ref{numberofcompletion}. 
To prove these results, we need some knowledges from algebraic geometry. For convenience, the details 
of the statements and their proofs are thus given in the last section of this paper.

In the rest of this section, we shall assume that the given entries are from a matrix of rank $r$. However, 
in general, we do not know the rank $r>0$ of $M$ in advance. 
Thus, we have to make a guess of $r$. Let $r_g$ be a guessed rank. 
As we know  any reasonable choice of $r_g$ must  satisfy $m> 2nr_g- r_g^2$,  
we still have either $r_g<r$, $r_g=r$ or $r_g>r$.  
Choose a correct rank $r_g=r$ is a key to have the AP Algorithm, i.e. 
Algorithm~\ref{alg1} converges with a linear convergence rate. Otherwise, the convergence rate may not 
be linear. That is, when $r_g=\hbox{rank}(M)$, 
we can show that  Algorithm~\ref{alg1} converge to $M_{r_g}$ linearly. 
Otherwise, when $r_g< \hbox{rank}(M)$,
Algorithm~\ref{alg1} converges to a matrix under some conditions and 
may not be the desired matrix $M$. Thus, this section
is divided into three parts. We shall discuss the two cases in the first two subsections and leave the
numerical results in the third subsection.  

Another important issue is the distribution of $\Omega\subset \{(i,j), i,j=1, \cdots, n\}$. Clearly, if a 
column of $M$ is completely missing, one is not able to recover this column no matter what kind of rank $r$
of $M$ is and how large $m=|\Omega|$ is. If we let $\bfx\in \mathbb{R}^{n^2-m}$ be the unknown entries 
of $M$, the determinant of the 
sub-matrix of any $r+1$ rows and $r+1$ columns of $M$ will be zero which forms a polynomial equation 
with coefficients formed from known entries $M|_{\Omega}$. We
have $n^2-m$ unknowns while ${n\choose r+1}^2$ submatrices from $M$ which will result in ${n\choose r+1}^2$ 
polynomial equations. Since we have $n^2-m < n^2 -2nr+ r^2=(n-r)^2$ unknowns and ${n\choose r+1}^2$ equations, 
the system of polynomial equations is overdetermined. 
We have to assume that the system is consistent, i.e. the system
has a solution. Otherwise, the overdetermined system has no solution, i.e. the matrix $M$ can not be 
completed. 
 Hence, for the rest of the paper, let us assume that the overdetermined system of polynomial equations have a solution, i.e. 
$M$ can be completed .

\subsection{Convergence of Algorithm~\ref{alg1} When 
$r_g=\hbox{Rank}(M)$}
We start with some preliminary results. 
\begin{lemma}\label{lemma1}
Let $L$ be a linear subspace of $\mathbb{R}^n$. Suppose $P_{L}$ denote 
the orthogonal projection onto $L$. Then, for any $x \in \mathbb{R}^n$
	$$
	\norm{x} = \norm{P_L(x)} \text{ if and only if } x \in L
	$$ 
	Equivalently,
	$$
	\norm{P_L(x)} < \norm{x} \text{ if and only if } x \not \in L
	$$
\end{lemma}
\begin{proof} The 'if' part is clear. 
So, let us prove the 'only if' part.

Let $l_1, l_2, \cdots l_k$ be a orthonormal basis of $L$. Extend it to 
a orthonormal basis $l_1,l_2, \cdots l_n$ of $\mathbb{R}^n$. 	Then, 
$$
	x = \sum_{i = 1}^{n} \langle x , l_i\rangle l_i
$$ 
and
$$
	\norm{x}^2 = \sum_{i = 1}^{n} { \langle x , l_i\rangle}^2 = \norm{P_L(x)}^2 + \sum_{i = k+1}^{n} { \langle x , l_i\rangle}^2
	$$
	Now it follows that if $\norm{x} = \norm{P_L(x)}$, then 
$\sum_{i = k+1}^{n} { \langle x , l_i\rangle}^2 = 0$, which 
implies $\langle x, l_i\rangle = 0$ for all $i \geq k+1$. Therefore, 
$x =  \sum_{i = 1}^{k} \langle x , l_i\rangle l_i \in L$.   
\end{proof}

\begin{lemma}\label{lemma2}
Let $L_1$ and $L_2$ be two linear subspaces of $\mathbb{R}^n$. Suppose 
$P_{L_1}$ and $P_{L_2}$ denote the orthogonal projection onto $L_1$ and 
$L_2$ respectively. Then, $L_1 \cap L_2 = \{0\}$ if and only if
\begin{equation}
	\norm{P_{L_2}P_{L_1}} < 1.
\label{onekey}
\end{equation}
\end{lemma}
\begin{proof}
Assume $L_1 \cap L_2 = \{0\}$. Let $x \neq 0 \in \mathbb{R}^n$. Then if 
$P_{L_1}(x)=0$, then $P_{L_2}P_{L_1}(x)=0<\norm{x}$. 
Otherwise, $P_{L_1}(x) \neq 0$. 
Since $L_1 \cap L_2 = \{0\}$, $P_{L_1}(x) \not \in L_2$. 
Therefore, using Lemma \ref{lemma1}, we get
$$
\norm{P_{L_2}P_{L_1}(x)} < \norm{P_{L_1}(x)} \leq \norm{P_{L_1}}\norm{x} \le \norm{x}. 
$$
Hence, we have 
$$
\norm{P_{L_2}P_{L_1}(x)} < \norm{x}
$$ 
for all non-zero $x \neq 0 \in \mathbb{R}^n$. 	So, 
$$
	\norm{P_{L_2}P_{L_1}} < 1.
$$
To prove the other direction, assume $\norm{P_{L_2}P_{L_1}} < 1$. 
Assume, on the contrary, that $L_1 \cap L_2 \neq \{0\}$. Let $x \neq 0 
\in L_1 \cap L_2$ be a nonzero vector in the intersection. Then 
$P_{L_2}P_{L_1}(x) = P_{L_2}(x) = x$ which implies that 
$\norm{P_{L_2}P_{L_1}(x)} = \norm{x}$, contradicting the assumption. 
\end{proof}

\begin{lemma}\label{gradientofproj}
Let $M \in \mathcal{M}_r$.  
Then the projection operator $P_{\mathcal{M}_r}$ is well 
defined (single-valued) 
in a neighborhood of $M$ and is differentiable with gradient
\begin{equation}
\label{key1}
\nabla P_{\mathcal{M}_r}(M) = \operatorname{P}_{T_{\mathcal{M}_r}(M)},
\end{equation} 
where $T_\mathcal{M}(M)$ is the tangent space of $\mathcal{M}$ at $M$ 
and $\operatorname{P}_{T_\mathcal{M}(M)}$ is the projection operator 
onto the tangent space. 
\end{lemma}
\begin{proof}
Since the projection $P_{\mathcal{M}_r}$	 of a matrix $X$ is 
obtained by hard thresholding the least $n-r$ 
singular values, we see that the projection is unique if 
$\sigma_{r}(M)\neq \sigma_{r+1}(M)\ge 0$. Now consider the 
neighborhood $V$ of $M$ given by 
$$
V := \left\{X \in \mathbb{R}^{n\times n} \mid \norm{X-M}_F < \frac{\sigma_{r}(M)}{4}\right\}.
$$ 
Then, by Weyl's  \cite{Weyl} or more generally Mirsky's \cite{Mirsky} 
perturbation bounds on singular values, we have
$$
\abs{\sigma_{r}(X)-\sigma_{r}(M)}\leq \norm{X-M}_F 
< \frac{\sigma_{r}(M)}{4}
$$  
and 
$$
\abs{\sigma_{r+1}(X)-\sigma_{r+1}(M)}\leq \norm{X-M}_F 
< \frac{\sigma_{r}(M)}{4}.
$$ 
Hence, noting $\sigma_{r+1}(M)=0$, we observe that 
$$
\sigma_{r+1}(X)<\frac{\sigma_{r}(M)}{4}<\frac{3\sigma_{r}(M)}{4}
<\sigma_{r}(X). 
$$ 
In particular, 
$$
\sigma_{r}(X)\neq \sigma_{r+1}(X). 
$$ 
Therefore, $P_{\mathcal{M}_r}$ is single valued in the neighborhood $V$.

For second part of the result, we refer to Theorem 25 
in \cite{FlorianPierre} which is stated below.  
We have changed the notations for ease of reading. In particular, note that although 
the $X$ has rank greater than $r$ in \cite{FlorianPierre}, its easy to see that their proof goes 
through when $X$ has rank greater than or equal to $r$. 
Intuitively, it is easy to see that the gradient
vector of the projection $P_{{\cal M}_r}$ of smooth manifold ${\cal M}_r$ at $M$ 
will be the projection onto the tangent plane $T_{{\cal M}_r}$ at $M$ in general. 
\end{proof}

The following results was used in the proof above. 
\begin{theorem}[F. Feppon and P.J. Lermusiaux, 2017\cite{FlorianPierre}]
Consider $X \in \mathbb{R}^{n \times m}$ with rank greater than $r$ and denote $X = \sum_{i 
= 1}^{r+k}\sigma_iu_iv_i^\top $ be its SVD decomposition, where the singular values are ordered decreasingly: 
$\sigma_1 \geq \sigma_2 \geq \cdots \sigma_{r+k}$. Suppose that the orthogonal projection $P_{\mathcal{M}_r}(X)$
of $X$ onto $\mathcal{M}_r$ is uniquely defined, that is $\sigma_{r}(X) > \sigma_{r+1}(X)$. Then 
$P_{\mathcal{M}_r}$, the SVD truncation operator of
order $r$, is differentiable at $X$ and the differential in a direction $Y$	is given by the formula
\begin{align}
\label{key2}
\nabla_Y P_{\mathcal{M}_r}(X) =& \operatorname{P}_{T_{\mathcal{M}_r}(P_{\mathcal{M}_r}(X))}(Y) \cr  
&+ \sum_{1 \leq i \leq r\atop 1 \leq j \leq k}
\left[ \frac{\sigma_{r+j}}{\sigma_i - \sigma_{r+j}} 
\langle Y, \Phi^{+}_{i,r+j}\rangle \Phi^{+}_{i,r+j} - \frac{\sigma_{r+j}}{\sigma_i + \sigma_{r+j}} 
\langle Y, \Phi^{-}_{i,r+j}\rangle \Phi^{-}_{i,r+j}  \right],
\end{align}
where 
 $$
 \Phi^{\pm}_{i,r+j} = \frac{1}{\sqrt{2}}(u_{r+j}v_i^\top  \pm u_iv_{r+j}^\top ) 
 $$
 are the principal directions corresponding to the principal curvature of the manifold of rank-$r$ matrices.
\end{theorem}
\begin{proof}
Refer to Theorem 25 in \cite{FlorianPierre}. 
\end{proof}


We are now ready to establish the convergence of Algorithm~\ref{alg1} under a sufficient condition. 
\begin{theorem}
\label{convergencethm}
Assume $T_{\mathcal{A}_{\Omega}}(M)\cap T_{\mathcal{M}_r}(M) = \{0\}$. Then Algorithm~\ref{alg1} converges to $M$ 
locally at a linear rate, i.e. there exists a neighborhood $V$ around $M$ such that if $X_0 \in V$, then there 
exists a positive constant $c < 1$ such that 
\begin{equation}
\label{lineconv}
\norm{X_k-M} < c^k\norm{X_0-M},
\end{equation} 
where $X_k$ is the $k$th iteration from Algorithm 1. 
\end{theorem}
\begin{proof}
For notational convenience, let 
	$$
	f(X) := P_{\mathcal{A}_{\Omega}}(P_{\mathcal{M}_r}(X)).
	$$
Note that $\mathcal{A}_\Omega$ is an affine space, the gradient $\nabla P_{\mathcal{A}_\Omega}$ of
the projection $P_{\mathcal{A}_\Omega}$ is the projection onto the tangent space of the affine space
$\mathcal{A}_\Omega$.  
By Lemma \ref{gradientofproj} and chain rule, we have
$$
(\nabla f) (X) = P_{T_{\mathcal{A}_{\Omega}}(M)}(\operatorname{P}_{T_{\mathcal{M}_r}(M)}(X)).
$$
as $T_{\mathcal{A}_{\Omega}}(M)=T_{\mathcal{A}_{\Omega}}(X)$ for all $X$. 
	
	Now from the definition of differentiability of $f$ at $M$, we have
	$$
	\lim\limits_{X \rightarrow M}\frac{\norm{f(X)-f(M) - \nabla f (M) \cdot (X-M)}}{\norm{X-M}} = 0.
	$$
Hence, there exist an open ball $V$, say a ball $V = B_{r_0}(M)$ centered at $M$ of radius $r_0$ around $M$ 
such that, for all $X \in V$ 
$$
		\frac{\norm{f(X)-f(M) - \nabla f(M) \cdot (X-M)}}{\norm{X-M}} < \epsilon,
$$
where $\epsilon = \frac {1-\norm{\nabla f}}{2}>0$. Using our hypothesis and Lemma \ref{lemma2}, we 
have 
$\norm{\nabla f(M)} = \norm{P_{T_{\mathcal{A}_{\Omega}}(M)}\operatorname{P}_{T_{\mathcal{M}_r}(M)}}<1$.   
Therefore, for all $X \in V$, we use  $M=f(M)$ to have  
\begin{align*}
		\norm{f(X)-M} &= \norm{f(X)-f(M)} \cr
& \leq \norm{f(X)-f(M) - \nabla f \cdot (X-M)} + \norm{\nabla f(M) \cdot (X-M)}\\
		& < \epsilon \norm{X-M} + \norm{\nabla f(M)} \norm{(X-M)}\\
		& = (\epsilon + \norm{\nabla f(M)})\norm{X-M}\\
		& \leq \frac {1+ \norm{\nabla f(M)}}{2}\norm{X-M}.
\end{align*}
where $\frac {1+ \norm{\nabla f(M)}}{2} < 1$ since $\norm{\nabla f (M)}<1$ as discussed above. 

Setting $c = \frac {1+ \norm{\nabla f(M)}}{2}<1$, we can rewrite the above inequality as follows: 
	\begin{equation}
		\norm{f(X)-M} < c \norm{X-M} \text{ for all } X \in V.
	\end{equation}

Hence, if $X_k \in V =  B_{r_0}(M)$, we use $X_{k+1}=f(X_k)$ to have 	
	$$
	\norm{X_{k+1}-M} = \norm{f(X_k)-M} < c \norm{X_k-M} \leq r_0
	$$
	which implies $X_{k+1} \in V =  B_{r_0}(M)$. 
	So, if the initial guess $X_0 \in V$, we have, by induction,
	$$
	X_k \in V \text{ for all } k
	$$
	and
	$$
	\norm{X_k-M} \leq c^k\norm{X_0-M}.
	$$
We have thus completed the proof. 
\end{proof}

We will now derive certain equivalent conditions for hypothesis of the above theorem viz. $T_{\mathcal{A}_{\Omega}}(M)\cap T_{\mathcal{M}_r}(M) = \{0\}$. 
Let us recall the following property which is known in the literature. For convenience, we include a proof. 
\begin{lemma}
\label{key1}
The tangent space $T_{\mathcal{M}_r}(M)$ has an explicit description as follows: 
\begin{equation}
\label{eq:key1}
T_{\mathcal{M}_r}(M) = \left\{XM + MY \mid X \in \mathbb{R}^{n \times n} \text{ and } Y \in \mathbb{R}^{n \times n} \right\}.
\end{equation}
\end{lemma}
\begin{proof}
First recall that the tangent space $T_{\mathcal{M}_r}(M)$ to a manifold $\mathcal{M}_r$ at a point $M$ is the linear space spanned by all the tangent vectors at $0$ to smooth curves $\gamma: \mathbb{R} \rightarrow \mathcal{M}_r$ such that $\gamma(0) = M$.

Now let $M \in \mathcal{M}_r$ be a $n \times n$ matrix of rank $r$. We can write $M = X_0Y_0^\top $ where $X_0,Y_0 \in \mathbb{R}^{n \times r}$ and both $X_0$ and $Y_0$ have full column rank. This is possible because $M$ has exactly rank $r$.

Let $\gamma(t) = X(t)Y(t)^\top $ be a smooth curve such that $X(0) = X_0$ and $Y(0)=Y_0$. Hence, 
$\gamma(0) = X_0Y_0^\top = M$. Since $X_0$ and $Y_0$ have full column rank, $X_0$ and $Y_0$ have a $r\times r$ 
minor that does not vanish. Since nonvanishing of a minor is an open condition, 
there exist an open neighbourhood of $M$ to which if we restrict the curve $\gamma$, we can assume $X(t)$ and 
$Y(t)$ have full column rank. In other words, we can assume, without loss of generality, that $X(t)^\top X(t)$ 
and $Y(t)^\top Y(t)$ are invertible $r \times r$ matrices for all $t$.

By product rule, we obtain
\begin{align*}
	\dot{\gamma}(0) & = \dot{X}(0)Y(0)^\top  +  X(0)\dot{Y}(0)^\top \\
	& =  \dot{X}(0)Y_0^\top  + X_0\dot{Y}(0)^\top \\ 
	& = \dot{X}(0)(X_0^\top X_0)^{-1}(X_0^\top X_0)Y_0^\top  + X_0(Y_0^\top Y)(Y_0^\top Y_0)^{-1}\dot{Y}(0)^\top \\
	& = \left(\dot{X}(0)(X_0^\top X_0)^{-1}X_0^\top \right)(X_0 Y_0^\top ) + (X_0Y_0^\top )\left(Y_0(Y_0^\top Y)^{-1}\dot{Y}(0)^\top \right)\\
	& = \left(\dot{X}(0)(X_0^\top X_0)^{-1}X_0^\top \right)M + 
M\left(Y_0(Y_0^\top Y_0)^{-1}\dot{Y}(0)^\top \right)\\
	& \in  \left\{XM + MY \mid X \in \mathbb{R}^{n \times n} \text{ and } Y \in \mathbb{R}^{n \times n} \right\}.
\end{align*}

Now to prove the reverse inclusion, 
 let $AM + MB \in  \left\{XM + MY \mid X \in \mathbb{R}^{n \times n} \text{ and } Y \in \mathbb{R}^{n \times n} \right\}$. Consider the smooth curve $\gamma(t) = X(t)Y(t)^\top $ defined by  
$$
X(t) = t(AX_0) + X_0
$$
and 
$$
Y(t) = t\left((Y_0B)^\top \right) + Y_0.
$$
An easy computation shows that $\gamma(0) = M$ and $\dot{\gamma}(0) = AM + MB$.
Hence we get the equality

\begin{equation*}
T_{\mathcal{M}_r}(M) = \left\{XM + MY \mid X \in \mathbb{R}^{n \times n} \text{ and } Y \in \mathbb{R}^{n \times n} \right\}
\end{equation*}
This completes the proof. 
\end{proof} 
One can consider $T_{\mathcal{M}_r}(M)$ as a linear space in $\mathbb{R}^{n^2}$ by rewriting it as 
$$
T_{\mathcal{M}_r}(M) \cong \operatorname{Range}(T_M) = \left\{ T_M \cdot \begin{bmatrix}
{(X^1)}^\top \\
\vdots\\
{(X^n)}^\top \\
Y_1\\
\vdots\\
Y_n\\
\end{bmatrix} \mid  X \in \mathbb{R}^{n \times n} \text{ and } Y \in \mathbb{R}^{n \times n}  \right\} 
$$
where $T_M$ is a block matrix of size $n^2 \times 2n^2$ consisting of $2n^3$ blocks of size $1 \times n$, $X^i$ and $X_j$ denotes the $i^{th}$ row and $j^{th}$ column of a matrix $X$ respectively.
%

Explicitly, $T_M$ would take the form
	\[
		T_M = \left[
		\begin{array}{c|c|c|c|c|c|c|c|c|c|c|c|c|c|c}
		M_1^\top  & 0 & \cdots & 0 & \cdots & \cdots & \cdots & 0 & M^1 & 0 & \cdots & \cdots & \cdots & \cdots & 0\\ 
		\hline
		\vdots & \vdots & \vdots & \vdots & \vdots & \vdots & \vdots & \vdots &\vdots & \vdots & \vdots &\vdots &\vdots &\vdots &\vdots \\  
		\hline
		0 & 0 & \cdots & 0 & M_j^\top  & 0 & \cdots & 0 & 0 & \cdots & 0 & M^i & 0 & \cdots & 0\\ 
		\hline
		\vdots & \vdots & \vdots & \vdots & \vdots & \vdots & \vdots & \vdots &\vdots & \vdots & \vdots &\vdots &\vdots &\vdots &\vdots \\
		\end{array}		
		\right]
	\]
where the each row corresponds to each index in $\{1,2,\cdots, n\}\times \{1,2,\cdots, n\}$. 

Let $T^{\Omega}_M$ and $T^{\Omega^c}_M$ denote the matrix obtained from $T_M$ by choosing the rows 
corresponding to $\Omega$ and $\Omega^c$, respectively.

\begin{example}
Suppose $M = \begin{bmatrix}
	1 & 4\\
	2 & 8\\
	\end{bmatrix}$ and $\Omega = \{(1,2),(2,1) \}$. Then
	$$
	T^{\Omega}_M = \left[
	\begin{array}{cc|cc|cc|cc}
	4 & 8 & 0 & 0 & 0 & 0 & 1 & 4\\
	\hline
	0 & 0 & 1 & 2 & 2 & 8 & 0 & 0\\ 
	\end{array}
	\right]
	$$
	
	$$
	T^{\Omega^c}_M = \left[
	\begin{array}{cc|cc|cc|cc}
	1 & 2 & 0 & 0 & 1 & 4 & 0 & 0\\
	\hline
	0 & 0 & 4 & 8 & 0 & 0 & 2 & 8\\ 
	\end{array}
	\right]
	$$
	and
	$$
	T_M = \left[
	\begin{array}{cc|cc|cc|cc}
		4 & 8 & 0 & 0 & 0 & 0 & 1 & 4\\
	\hline
	0 & 0 & 1 & 2 & 2 & 8 & 0 & 0\\
	\hline
	1 & 2 & 0 & 0 & 1 & 4 & 0 & 0\\
	\hline
	0 & 0 & 4 & 8 & 0 & 0 & 2 & 8\\ 
	\end{array}
	\right]
	$$
	
\end{example}

\begin{example}
	Suppose $M = \begin{bmatrix}
	-3 & -1 & -4\\
	9 & 3 & 12\\
	6 & 2 & 8\\
	\end{bmatrix}$ and $\Omega = \{(1,1),(1,3),(2,2),(3,1) \}$, Then
	$$
	T^{\Omega}_M = \left[
	\begin{array}{ccc|ccc|ccc|ccc|ccc|ccc}
	-3 & 9 & 6 & 0 & 0 & 0 & 0 & 0 & 0 & -3 & -1 & -4 & 0 & 0 & 0 & 0 & 0 & 0\\
	\hline
	-4 & 12 & 8 & 0 & 0 & 0 & 0 & 0 & 0 & 0 & 0 & 0 & 0 & 0 & 0 & -3 & -1 & -4\\
	\hline
	0 & 0 & 0 & -1 & 3 & 2 & 0 & 0 & 0 & 0 & 0 & 0  & 9 & 3 & 12 & 0 & 0 & 0\\
	\hline
	0 & 0 & 0 & 0 & 0 & 0 & -3 & 9 & 6 & 6 & 2 & 8 & 0 & 0 & 0 & 0 & 0 & 0\\
	\end{array}
	\right]
	$$
	
		$$
	T^{\Omega^c}_M = \left[
	\begin{array}{ccc|ccc|ccc|ccc|ccc|ccc}
	-1 & 3 & 2 & 0 & 0 & 0 & 0 & 0 & 0 & 0 & 0 & 0 & -3 & -1 & -4 & 0 & 0 & 0\\
	\hline
	0 & 0 & 0 & -3 & 9 & 6 & 0 & 0 & 0 & 9 & 3 & 12 & 0 & 0 & 0 & 0 & 0 & 0\\
	\hline
	0 & 0 & 0 & -4  & 12 & 8 & 0 & 0 & 0 & 0 & 0 & 0  & 0 & 0 & 0 & 9 & 3 & 12\\
	\hline
	0 & 0 & 0 & 0 & 0 & 0 & -1 & 3 & 2 & 0 & 0 & 0 & 6 & 2 & 8 & 0 & 0 & 0\\
	\hline
	0 & 0 & 0 & 0 & 0 & 0 & -4 & 12 & 8 & 0 & 0 & 0 & 0 & 0 & 0 & 6 & 2 & 8\\
	\end{array}
	\right]
	$$
	
	and
	$$
	T_M = \left[
	\begin{array}{ccc|ccc|ccc|ccc|ccc|ccc}
	-3 & 9 & 6 & 0 & 0 & 0 & 0 & 0 & 0 & -3 & -1 & -4 & 0 & 0 & 0 & 0 & 0 & 0\\
	\hline
	-4 & 12 & 8 & 0 & 0 & 0 & 0 & 0 & 0 & 0 & 0 & 0 & 0 & 0 & 0 & -3 & -1 & -4\\
	\hline
	0 & 0 & 0 & -1 & 3 & 2 & 0 & 0 & 0 & 0 & 0 & 0  & 9 & 3 & 12 & 0 & 0 & 0\\
	\hline
	0 & 0 & 0 & 0 & 0 & 0 & -3 & 9 & 6 & 6 & 2 & 8 & 0 & 0 & 0 & 0 & 0 & 0\\
	\hline
	-1 & 3 & 2 & 0 & 0 & 0 & 0 & 0 & 0 & 0 & 0 & 0 & -3 & -1 & -4 & 0 & 0 & 0\\
	\hline
	0 & 0 & 0 & -3 & 9 & 6 & 0 & 0 & 0 & 9 & 3 & 12 & 0 & 0 & 0 & 0 & 0 & 0\\
	\hline
	0 & 0 & 0 & -4  & 12 & 8 & 0 & 0 & 0 & 0 & 0 & 0  & 0 & 0 & 0 & 9 & 3 & 12\\
	\hline
	0 & 0 & 0 & 0 & 0 & 0 & -1 & 3 & 2 & 0 & 0 & 0 & 6 & 2 & 8 & 0 & 0 & 0\\
	\hline
	0 & 0 & 0 & 0 & 0 & 0 & -4 & 12 & 8 & 0 & 0 & 0 & 0 & 0 & 0 & 6 & 2 & 8\\
	\end{array}
	\right]
	$$
\end{example}

Next we need 
\begin{lemma} 
\label{key2}
The tangent space $T_{\mathcal{A}_{\Omega}}(M)$ at $M$ can be given explicitly as follows. 
\begin{equation}
T_{\mathcal{A}_{\Omega}}(M) = \left\{X \in \mathbb{R}^{n \times n} \mid P_\Omega(X) = 0  \right\}.
\end{equation}
\end{lemma}
\begin{proof}
Recall that $$
\mathcal{A}_M := \left\{X \mid P_\Omega(X-M) = 0 \right\}.
$$
Since $P_\Omega(X-M) = P_\Omega(X) - P_\Omega(M) = P_\Omega(X) - P_\Omega(P_\Omega(M)) = P_\Omega(X -P_\Omega(M))$, we get that the set $\mathcal{A}_\Omega$ is a translation of the linear space $\left\{X \in \mathbb{R}^{n \times n} \mid P_\Omega(X) = 0  \right\}$ by $P_\Omega(M)$, i.e.
$$
\mathcal{A}_\Omega = \left\{X \in \mathbb{R}^{n \times n} \mid P_\Omega(X) = 0  \right\} + P_\Omega(M)
$$
Hence we have that the tangent space of $\mathcal{A}_\Omega$ at $M$ is equal to the tangent space of the vector space 
$\left\{X \in \mathbb{R}^{n \times n} \mid P_\Omega(X) = 0  \right\}$ at $M-P_\Omega(M)$. But the tangent space of a vector space at any point is the vector space itself. Hence the result follows. 
\end{proof}
With the above preparation, we have another main result in this section. 
\begin{theorem}
\label{LVmain}
The following statements are equivalent:\\
\begin{enumerate}
\item $T_{\mathcal{A}_{\Omega}}(M)\cap T_{\mathcal{M}_r}(M) = \{0\}$
\vspace{0.1in}
\item $\operatorname{Rowspace}\left(T^{\Omega^c}_M\right) \subseteq 
\operatorname{Rowspace}\left(T^{\Omega}_M\right)$ 
\vspace{0.1in}
\item $\operatorname{Rank}\left(T^{\Omega}_M \right) = 2nr-r^2$, 
where $r = \operatorname{Rank}(M)$\vspace{0.1in}
\item The matrix $V^{\Omega}(M)$ of size $\abs{\Omega} \times 
\abs{\Omega}$ defined by
\begin{equation}
\label{varghese}
V_{(i_1,j_1),(i_2,j_2)}^{\Omega}(M) = 
\begin{cases}
0 &  i_1 \neq i_2 \text{ and } j_1 \neq j_2\\
\langle M_{j_1},M_{j_2}  \rangle  & i_1 = i_2 \text{ and } j_1 \neq j_2\\
\langle M^{i_1},M^{i_2}  \rangle  & i_1 \neq i_2 \text{ and } j_1 = j_2\\
\norm{M^{i_1}}^2 + \norm{M_{j_1}}^2 & i_1 = i_2 \text{ and } j_1 = j_2
		\end{cases}
\end{equation}
has rank $2nr-r^2$, where $M_j$ stands for the $j$th column and $M^i$ for the $i$th row of $M$.
\end{enumerate}
\end{theorem}
\begin{proof}
$(1)\iff (2)$ Note that the elements of $T_{\mathcal{A}_{\Omega}}(M)\cap T_{\mathcal{M}_r}(M)$ consists of matrices of the form $XM + MY$ such that the elements in positions $\Omega$ is zero by Lemmas~\ref{key1} 
and \ref{key2}. Hence, observing that $T_{\mathcal{M}_r}(M)$ can be considered as the range of $T_M$ 
and that the rows of $T_M$ correspond to each index in $\{1,2,\cdots,n\} \times \{1,2,\cdots,n\}$, we can conclude that $T_{\mathcal{A}_{\Omega}}(M)\cap T_{\mathcal{M}_r}(M) = \{0\}$ if and only if
$$
	\operatorname{NullSpace}\left(T^{\Omega}_M\right) \subseteq \operatorname{NullSpace}\left(T^{\Omega^c}_M\right)
	$$
	which is equivalent to
	$$
	\operatorname{NullSpace}\left(T^{\Omega}_M\right)^\perp \supseteq \operatorname{NullSpace}\left(T^{\Omega^c}_M\right)^\perp
	$$
	The result follows by noting that 
	$$
	\operatorname{Rowspace}\left(T^{\Omega^c}_M\right) = \operatorname{NullSpace}\left(T^{\Omega^c}_M\right)^\perp
	$$
	and
	$$
\operatorname{Rowspace}\left(T^{\Omega}_M\right) =\operatorname{NullSpace}\left(T^{\Omega}_M\right)^\perp.
	$$
	
$(2)\iff (3)$ 
We begin by recalling that dimension of a tangent space is equal to dimension of the manifold. So, $\dim(T_{\mathcal{M}_r}(M)) = 2nr-r^2$.
Now 
$$
2nr-r^2 = \dim(T_\mathcal{M}(M)) = \dim(Range(T_M)) = \operatorname{Rank}(T_M) = 
\operatorname{Rank}\left(\operatorname{Rowspace}\left(T_M\right)\right).
$$
Now the equivalence $(2)\iff (3)$ follows by recalling that $T^{\Omega}_M$ and $T^{\Omega^c}_M$ were obtained from $T_M$ by choosing the rows 
corresponding to $\Omega$ and $\Omega^c$, respectively
	
$(3)\iff (4)$ 
The equivalence follows from fact that $V^{\Omega}(M) = T^{\Omega}_M \left(T^{\Omega}_M\right)^\top $. 
Hence $\operatorname{Rank}(V^{\Omega}(M)) = \operatorname{Rank}(T^{\Omega}_M).$
\end{proof} 

In general, the rank of $V^\Omega(M)$ is less 
than or equal to $2nr-r^2$. The equality occurs when the tangent spaces intersect trivially.  
The following example is an illustration of the linear convergence of 
the error when the condition $T_{\mathcal{A}_{\Omega}}(M)\cap 
T_{\mathcal{M}_r}(M) = \{0\}$ is satisfied.

\begin{example}
We find a  $15 \times 15$ matrix $M$ of rank 2 which has $28\%$ of entries missing. 
A straightforward computation shows that $\operatorname{Rank}(V^\Omega(M)) = 2nr-r^2$. Hence, $M$ satisfies the condition $T_{\mathcal{A}_{\Omega}}(M)\cap T_{\mathcal{M}_r}(M) = \{0\}$. Hence, by Theorems~\ref{LVmain}
and \ref{convergencethm}, we know that Algorithm~\ref{alg1} will converge in a linear fashion. 

$$
M=
$$
$$
\tiny  
\left[\begin{array}{ccccccccccccccc} 0.3474 & 0.0897 & 0.3971 & 0.4644 & 0.4168 & 0.7576 & 0.8206 & 0.8161 & 0.3279 & 0.3851 & 0.0825 & 0.4742 & 0.7684 & 0.6113 & 0.3832 \\
	0.1502 & 0.0414 & 0.2196 & 0.2450 & 0.2731 & 0.4415 & 0.4293 & 0.4358 & 0.1859 & 0.1574 & 0.0493 & 0.2386 & 0.4502 & 0.3087 & 0.1999 \\
	0.3853 & 0.1079 & 0.5912 & 0.6542 & 0.7544 & 1.1986 & 1.1445 & 1.1660 & 0.5024 & 0.3985 & 0.1343 & 0.6315 & 1.2231 & 0.8176 & 0.5325 \\
	0.2174 & 0.0577 & 0.2760 & 0.3160 & 0.3141 & 0.5394 & 0.5562 & 0.5582 & 0.2305 & 0.2358 & 0.0594 & 0.3160 & 0.5484 & 0.4080 & 0.2594 \\
	0.2124 & 0.0493 & 0.1453 & 0.1940 & 0.0662 & 0.2317 & 0.3503 & 0.3303 & 0.1109 & 0.2539 & 0.0228 & 0.2216 & 0.2302 & 0.2835 & 0.1647 \\
	0.1026 & 0.0238 & 0.0701 & 0.0936 & 0.0318 & 0.1117 & 0.1691 & 0.1594 & 0.0535 & 0.1227 & 0.0110 & 0.1070 & 0.1110 & 0.1368 & 0.0795 \\
	0.2429 & 0.0600 & 0.2290 & 0.2798 & 0.1972 & 0.4141 & 0.4982 & 0.4864 & 0.1846 & 0.2785 & 0.0439 & 0.2974 & 0.4176 & 0.3823 & 0.2332 \\
	0.3848 & 0.0895 & 0.2658 & 0.3538 & 0.1248 & 0.4257 & 0.6385 & 0.6028 & 0.2032 & 0.4595 & 0.0421 & 0.4033 & 0.4232 & 0.5159 & 0.3002 \\
	0.3698 & 0.1015 & 0.5311 & 0.5943 & 0.6536 & 1.0640 & 1.0419 & 1.0562 & 0.4488 & 0.3894 & 0.1186 & 0.5806 & 1.0845 & 0.7511 & 0.4852 \\
	0.3631 & 0.0880 & 0.3138 & 0.3919 & 0.2395 & 0.5511 & 0.7003 & 0.6776 & 0.2496 & 0.4217 & 0.0575 & 0.4246 & 0.5540 & 0.5451 & 0.3282 \\
	0.2081 & 0.0480 & 0.1369 & 0.1850 & 0.0542 & 0.2139 & 0.3347 & 0.3141 & 0.1036 & 0.2498 & 0.0208 & 0.2133 & 0.2120 & 0.2726 & 0.1575 \\
	0.5203 & 0.1334 & 0.5792 & 0.6812 & 0.5942 & 1.0977 & 1.2049 & 1.1953 & 0.4769 & 0.5797 & 0.1192 & 0.6992 & 1.1126 & 0.9011 & 0.5627 \\
	0.4871 & 0.1231 & 0.5111 & 0.6090 & 0.4961 & 0.9538 & 1.0797 & 1.0652 & 0.4178 & 0.5487 & 0.1028 & 0.6328 & 0.9651 & 0.8148 & 0.5047 \\
	0.0287 & 0.0122 & 0.1183 & 0.1173 & 0.2001 & 0.2658 & 0.2007 & 0.2154 & 0.1057 & 0.0156 & 0.0311 & 0.0991 & 0.2738 & 0.1297 & 0.0927 \\
	0.2602 & 0.0617 & 0.1997 & 0.2577 & 0.1230 & 0.3353 & 0.4628 & 0.4422 & 0.1558 & 0.3070 & 0.0341 & 0.2867 & 0.3353 & 0.3673 & 0.2173 \end{array}  \right]
$$
and 
$$
M_\Omega =
$$
$$
\tiny 
 \left[\begin{array}{ccccccccccccccc} 0      & 0.0897 & 0.3971 & 0      & 0.4168 & 0.7576 & 0.8206 & 0      & 0.3279 & 0.3851 & 0.0825 & 0      & 0.7684 & 0.6113 & 0.3832 \\
	0.1502 & 0      & 0      & 0.2450 & 0      & 0.4415 & 0.4293 & 0.4358 & 0      & 0.1574 & 0      & 0.2386 & 0.4502 & 0      & 0.1999 \\
	0.3853 & 0.1079 & 0.5912 & 0.6542 & 0.7544 & 1.1986 & 0      & 1.1660 & 0.5024 & 0      & 0.1343 & 0      & 1.2231 & 0.8176 & 0.5325 \\
	0.2174 & 0.0577 & 0.2760 & 0      & 0.3141 & 0      & 0      & 0.5582 & 0.2305 & 0      & 0.0594 & 0.3160 & 0.5484 & 0.4080 & 0      \\
	0      & 0      & 0.1453 & 0.1940 & 0.0662 & 0.2317 & 0.3503 & 0      & 0.1109 & 0      & 0      & 0.2216 & 0      & 0      & 0.1647 \\
	0.1026 & 0.0238 & 0.0701 & 0.0936 & 0.0318 & 0.1117 & 0.1691 & 0.1594 & 0.0535 & 0.1227 & 0.0110 & 0      & 0.1110 & 0.1368 & 0.0795 \\
	0.2429 & 0.0600 & 0.2290 & 0.2798 & 0      & 0      & 0.4982 & 0      & 0.1846 & 0      & 0      & 0.2974 & 0.4176 & 0.3823 & 0.2332 \\
	0      & 0      & 0      & 0.3538 & 0.1248 & 0      & 0      & 0      & 0      & 0.4595 & 0      & 0      & 0      & 0.5159 & 0.3002 \\
	0      & 0.1015 & 0.5311 & 0.5943 & 0.6536 & 1.0640 & 1.0419 & 1.0562 & 0.4488 & 0.3894 & 0      & 0.5806 & 1.0845 & 0      & 0      \\
	0.3631 & 0.0880 & 0.3138 & 0      & 0.2395 & 0.5511 & 0.7003 & 0.6776 & 0.2496 & 0      & 0.0575 & 0.4246 & 0.5540 & 0.5451 & 0      \\
	0.2081 & 0.0480 & 0.1369 & 0.1850 & 0.0542 & 0.2139 & 0      & 0      & 0.1036 & 0.2498 & 0.0208 & 0.2133 & 0      & 0.2726 & 0.1575 \\
	0      & 0.1334 & 0.5792 & 0      & 0.5942 & 1.0977 & 1.2049 & 1.1953 & 0      & 0      & 0.1192 & 0.6992 & 0      & 0.9011 & 0.5627 \\
	0.4871 & 0      & 0.5111 & 0.6090 & 0      & 0.9538 & 1.0797 & 0      & 0      & 0.5487 & 0.1028 & 0.6328 & 0.9651 & 0.8148 & 0.5047 \\
	0.0287 & 0.0122 & 0.1183 & 0.1173 & 0.2001 & 0.2658 & 0.2007 & 0.2154 & 0      & 0.0156 & 0      & 0.0991 & 0.2738 & 0.1297 & 0.0927 \\
	0.2602 & 0.0617 & 0.1997 & 0.2577 & 0.1230 & 0.3353 & 0.4628 & 0      & 0.1558 & 0.3070 & 0.0341 & 0.2867 & 0.3353 & 0.3673 & 0.2173 \end{array}  \right]
$$
where $0$ stands for the unknown entries. 

\begin{figure}[h]
\centering
	\includegraphics[scale=0.5]{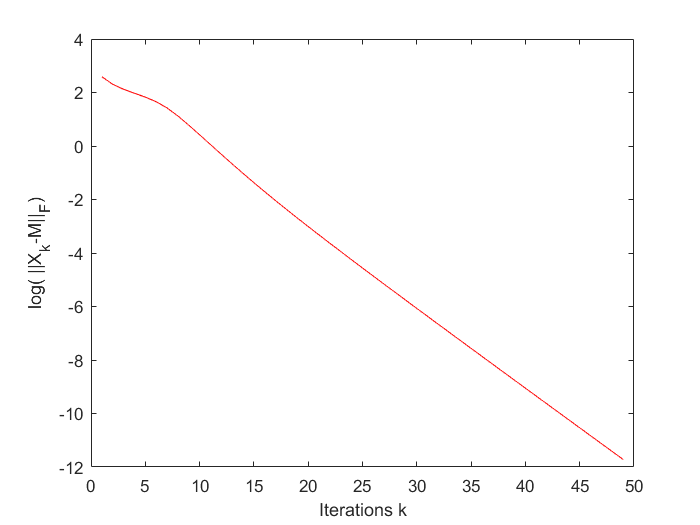}
\caption{Linear Convergence of the Iterations from Algorithm~\ref{alg1}\label{linearconvergence}}
	\end{figure}
	
Notice from the graph in Figure~\ref{linearconvergence} 
that as the iterations progress, the $X_k$ would eventually land in a 
neighborhood of $M$ where the convergence become linear.
\end{example}

The construction of $V^\Omega(M)$ enables us to choose $\Omega$ such that $V^\Omega$ is of full rank.  
We end with this subsection with the following
\begin{corollary}
Given $M$ with rank $r$, for any integer $m$ such that $2nr-r^2 \leq m \leq n^2$, there exists a subset $\Omega$ with $m=|\Omega|$
such that $V^\Omega$ is of full rank, equivalently $T_{\mathcal{A}_{\Omega}}(M)\cap T_{\mathcal{M}_r}(M) = \{0\}$ and Algorithm~\ref{alg1} can find $M$ in a linear fashion for a good
initial guess.
\end{corollary}
\begin{proof}
We mainly choose $\Omega$ such that the corresponding rows of $T_M$ which form $T_M^\Omega$ of rank $2nr-r^2$. Then Theorems~\ref{LVmain} and \ref{convergencethm} can be applied. 
\end{proof}

\subsection{Convergence of Algorithm~\ref{alg1} When 
$r_g\not=\hbox{Rank}(M)$}
In this subsection, we show that the algorithm does converge 
under certain reasonable assumption irrespective 
of whether our guessed rank $r_g$ is same as the rank $r$ of matrix 
$M$ or not. We begin with two trivial results. 
\begin{lemma}
\label{fact1}
Let $Y_k$ and $X_{k+1}$ be the matrices we obtain in the step 1 and step 2 of the $k^{th}$ iteration of Algorithm~\ref{alg1}. Then 
$$
	X_{k+1} = \begin{cases}
	(Y_k)_{i,j} & \text{if } (i,j) \not \in \Omega\\
	M_{i,j}  & \text{Otherwise.}
	\end{cases}
$$
That is, $X_{k+1}$ is the orthogonal projection of $Y_k $ onto $\mathcal{A}_\Omega$.
\end{lemma}

\begin{lemma}
\label{fact2}
Let  $X_{k+1} = \mathbf{U}\mathbf{\Sigma}\mathbf{V}^\top $ be the standard singular value decomposition 
with $\mathbf{\Sigma} = \operatorname{diag}\{\sigma_1,\cdots,\sigma_n\}$. Then  
$$
Y_{k+1} = \mathbf{U}\tilde{\mathbf{\Sigma}}\mathbf{V}^\top , 
$$
where $\tilde{\mathbf{\Sigma}} = diag\{\sigma_1,\cdots, \sigma_{r_g},0, \cdots,0\}$. 
	
Also $Y_{k+1}$ is the orthogonal projection of $X_{k+1} $ onto $\mathcal{M}_{r_g}$.  
\end{lemma}

$\mathcal{M}_{r_g}$, the collection  of $n \times n$ real (complex) matrices of rank ${r_g}$, forms a quasi-affine real (complex) variety and is a manifold of real (complex) dimension ${r_g}(2n-{r_g})$.

%

It is well known that $Y_k$, obtained from $X_k$ by SVD truncation, is the orthogonal projection of $X_k$ onto $\mathcal{M}_{r_g}$. Hence we $X_k-Y_k$ must be orthogonal to the tangent space of $\mathcal{M}_{r_g}$ at $Y_k$. Recall from earlier section that tangent space of $\mathcal{M}_{r_g}$ at the point $X$ is given by
$$
T_{\mathcal{M}_{r_g}}(X) = \left\{AX + XB, A \in \mathbb{R}^{m\times m}, B \in \mathbb{R}^{n\times n}\right\}
$$  
\begin{lemma}\label{fact11}
$Y_k$ satisfies: 
$$
	\langle A Y_k + Y_k B, X_k - Y_k\rangle = 0 \text{ for all } k,A \in \mathbb{R}^{n\times n} \mbox{ and } B \in \mathbb{R}^{n\times n}.
$$
\end{lemma}
\begin{proof}
Let $X_k = \mathbf{U}\mathbf{\Sigma}\mathbf{V}^\top $ and $Y_k = \mathbf{U}\tilde{\mathbf{\Sigma}}\mathbf{V}^\top $, where $\Sigma = 
\operatorname{diag}\{\sigma_1,\cdots,\sigma_n\}$ and $\tilde{\Sigma} =\operatorname{diag}\{\sigma_1,\cdots,\sigma_{r_g},0,\cdots,0\}$ be the 
singular value decompositions of $\pp k$ and $\qq k$ respectively. 
\begin{equation*}
\begin{split}
	\langle A Y_k + Y_k B, X_k - Y_k\rangle & = \operatorname{Trace}\left((\pp k - \qq k)^\top (A\qq k + \qq k B) \right)\\
	& =  \operatorname{Trace}\left({\bf V}{\bf {(\Sigma-\tilde{\Sigma})}}{\bf U}^\top  A \qq k \right) + \operatorname{Trace}\left({\bf V}{\bf {(\Sigma-\tilde{\Sigma})}}{\bf U}^\top  \qq k B  \right)\\
	& =  \operatorname{Trace}\left({\bf V}{\bf {(\Sigma-\tilde{\Sigma})}}{\bf U}^\top  A {\bf U} {\bf {\tilde{\Sigma}}} {\bf V}^\top  \right) + \operatorname{Trace}\left({\bf V}{\bf {(\Sigma-\tilde{\Sigma})}}{\bf U}^\top  {\bf U} {\bf {\tilde{\Sigma}}} {\bf V}^\top  B  \right)\\
	& = \operatorname{Trace}\left( {\bf V}^\top   {\bf V}{\bf {(\Sigma-\tilde{\Sigma})}}{\bf U}^\top  A {\bf U} {\bf {\tilde{\Sigma}}}  \right) + \operatorname{Trace}\left({\bf V}{\bf {(\Sigma-\tilde{\Sigma})}} {\bf {\tilde{\Sigma}}} {\bf V}^\top  B  \right)\\
	& = \operatorname{Trace}\left(  {\bf {\tilde{\Sigma}}} {\bf {(\Sigma-\tilde{\Sigma})}}{\bf U}^\top  A {\bf U}  \right) + \operatorname{Trace}\left({\bf V}{\bf {(\Sigma-\tilde{\Sigma})}} {\bf {\tilde{\Sigma}}} {\bf V}^\top  B  \right)\\
	& = 0.
\end{split}
\end{equation*} 
The last step uses the fact that $\tilde{\Sigma}(\Sigma - \tilde{\Sigma}) = (\Sigma - \tilde{\Sigma})\tilde{\Sigma} = 0$.
\end{proof}

From the definitions, it follows that 
\begin{equation}
\label{fact3}
	\norm{\pp k - \qq k} \geq \norm{\pp {k+1} - \qq k} \geq \norm{\pp {k+1} - \qq {k+1}} \text{ for all } k.
\end{equation}
From equation \eqref{fact3}, we observe that $\norm{\pp k - \qq k}$ is a non-increasing sequence bounded 
below by 0, it thus converges to its infimum. Thus, we have

\begin{lemma}\label{fact3.1}
The sequence 
	$$
	\norm{\pp k - \qq k}
	$$
converges.
\end{lemma}
So let 
\begin{equation}
\label{fact3.2}
L = \lim\limits_{k} \norm{\pp k - \qq k}^2.
\end{equation}
Next we have 
\begin{lemma}\label{fact4}
\begin{equation}
\label{eq:fact4}
\norm{\pp {k+1} - \pp k }^2 + \norm{\pp {k+1} - \qq k }^2 = \norm{\pp k - \qq k }^2
\end{equation}
\end{lemma}
\begin{proof}
The result (\ref{eq:fact4}) follows from Lemmas \ref{fact1} and \ref{fact2}. In fact  we have used 
the fact $\langle \pp {k+1} - \pp k, \pp {k+1} - \qq k \rangle = 0$ to have (\ref{eq:fact4}).   
\end{proof}


\begin{lemma}\label{fact6}
The series 
$$
	\sum_{k=1}^{\infty} \norm{\pp {k+1} - \pp k}^2
$$ 
converges. In particular 
$$
	\norm{\pp {k+1} - \pp k} \rightarrow 0. 
$$
\end{lemma}
\begin{proof}
We use (\ref{eq:fact4}) 
and (\ref{fact3}) to get 
$$
	\norm{\pp k - \qq k }^2 \geq \norm{\pp {k+1} - \pp k }^2 + \norm{\pp {k+1} - \qq {k+1} }^2 
$$
summing both sides from $k=1$ to $n$ we get	
$$
	\sum_{k=1}^{n} \norm{\pp k - \qq k }^2 \geq \sum_{k=1}^{n} \norm{\pp {k+1} - \pp k }^2  + \sum_{k=1}^{n} \norm{\pp {k+1} - \qq {k+1} }^2.  
$$
From which it follows that 
	$$
	\norm{\pp 1 - \qq 1}^2 \geq \norm{\pp n -\qq n}^2 + \sum_{k=1}^{n}\norm{\pp {k+1} - \pp k}^2 \geq \sum_{k=1}^{n}\norm{\pp {k+1} - \pp k}^2
	$$
Thus the partial sums of the $\sum_{k=1}^{\infty} \norm{\pp {k+1} - \pp k}^2$ forms an non-decreasing sequence bounded from above. 
The result follows immediately.
\end{proof}

\begin{lemma}\label{fact7}
The series 
	$$
	\sum_{k=1}^{\infty} \norm{(\pp k -\qq k)_{\Omega^c}}^2
	$$ 
converges. In particular 
	$$
	\norm{(\pp k -\qq k)_{\Omega^c}} \rightarrow 0.
	$$
\end{lemma}
\begin{proof}
\begin{equation*}
\begin{split}
	\norm{\pp {k+1} - \pp k}^2 & = \norm{(\pp {k+1} - \pp k)_\Omega}^2 + \norm{(\pp {k+1} - \pp k)_{\Omega^c}}^2\\
	 & = \norm{(\pp {k+1})_\Omega - (\pp k)_\Omega}^2 + \norm{(\pp {k+1})_{\Omega^c} - (\pp k)_{\Omega^c}}^2\\
	 & = \norm{M_\Omega - M_\Omega}^2 + \norm{(\pp {k+1})_{\Omega^c} - (\pp k)_{\Omega^c}}^2\\
	 & = \norm{(\pp {k+1})_{\Omega^c} - (\pp k)_{\Omega^c}}^2.    
	\end{split}
\end{equation*}
Now noting that $(\pp {k+1})_{\Omega^c} = (\qq k)_{\Omega^c}$ the above equation simplifies
\begin{equation*}
\begin{split}
	\norm{\pp {k+1} - \pp k}^2 & = \norm{(\qq {k})_{\Omega^c} - (\pp k)_{\Omega^c}}^2\\  
\end{split}
\end{equation*}
Summing both sides and using Lemma~\ref{fact6}, the result follows.	
\end{proof}


With the above preparation,  we are finally ready to establish the main convergence result in this 
subsection. 

\begin{theorem}\label{fact8}   There exist a subsequence of $(Y_k)_\Omega$ that converges, say without loss of generality, $(Y_k)_\Omega \rightarrow y^\star$. 
Assume that there are only finitely many rank-$r$ matrices $Y$ such that $P_\Omega(Y) = y^\star$.
Then there exist subsequences $\pp {k_j} $ and  $\qq {k_j}$ which converge, 
say  $Y^\star$ and $X^\star$ such that  
$$
\pp {k_j} \rightarrow X^\star \hbox{ and } \qq {k_j} \rightarrow Y^\star. 
$$ 
Furthermore, we have $X^\star|_{\Omega^c}=Y^\star|_{\Omega^c}$ and 
\begin{equation}
\label{mainresult2}
X^\star \in \mathcal{A}_\Omega \text{ and } 
\operatorname{rank}(Y^\star)\leq r_g. 
\end{equation}
\end{theorem}
\begin{proof} 
By Lemma~\ref{fact3.1},
$\norm{\pp k - \qq k} \rightarrow \sqrt{L}$, we see that the 
sequence $\norm{M_\Omega - (\qq k)_\Omega} = \norm{(\pp k)_\Omega - (\qq k)_\Omega}\le 2\sqrt{L}$ for all 
$k\ge 1$ without loss of generality. 
It follows that $\norm{(\qq k)_\Omega}, k\ge 1$ are a bounded sequence and hence, 
$\|(\qq k)_\Omega\|\le C_1<\infty$ for a positive constant $C_1$ and $(\qq k)_\Omega
\to y^*$ without loss of generality

Under the assumption that there are finitely many $Y\in \overline{{\cal M}_{r_g}}$ such that $P_\Omega(Y)=y^*$, 
we next claim that $Y_k, k\ge 1$ are bounded. Indeed, 
for any matrix $Y\in \overline{{\cal M}_{r_g}}$, the set of matrices with rank $\le r_g$, 
if we write the entries in $Y_{\Omega^c}$ 
as variables, say ${\bf x} \in \mathbb{R}^{n^2-m}$  while the entries $Y|_\Omega$ are known, 
the determinant of any $(r+1)\times (r+1)$ 
minor of $Y$ will be zero and is a polynomial function of variables ${\bf x}$ with coefficients 
based on the known entries $Y|_\Omega$. 
Thus, vanishing of all $(r+1)\times (r+1)$ minors would form a set of $({n\choose r_g+1})^2$ polynomial equations with variables 
${\bf x}$ and coefficients from entries in $Y|_\Omega$. 
 By our assumption, this set of polynomial equations have \emph{finitely} many solutions when the coeffficients of the system is derived from the $\Omega$ entries of $y^\star$.  Since the zeros of 
these polynomial equations are continuously dependent on the coefficients of polynomial functions, we see that there are finitely many solutions to the polynomial system when coefficients are derived from $(Y_k)_\Omega$ that are sufficiently close to $y^\star$ . We can
bound the zeros by using the coefficients. More precisely, these polynomial equations can 
be reduced to a triangular system (cf. \cite{CM12}),  that is,  writing ${\bf x}=(x_1, \cdots, x_{n^2-m})$
for a fixed order of these unknown entries, 
\begin{equation}
\label{triangular}
\begin{cases}
&f_1(x_1) = 0, \cr
&f_2(x_1,x_2) =0, \cr
& \cdots   \cdots ,\cr
&f_{n^2-m}(x_1, \cdots, x_{n^2-m})=0
\end{cases}
\end{equation}
for a set of polynomial functions $f_1, \cdots, f_{n^2-m}$ by using one of the computational methods 
discussed in  \cite{AM99}. Certainly, for each $k\ge 1$, these $f_i$ are dependent on $k$ in the sense
that the coefficients of $f_i$ are dependent on the values  $Y_k|_\Omega$.
Then we can use any standard bound of the zeros of univariate polynomials to find a bound 
of these variables ${\bf x}$ iteratively from 
the reduced system above. Indeed, the bound on $x_1$  of this system is obtained 
by $\max\{1, |a_i|,i=1, \cdots, r+1\}$ with coefficients 
$a_i$ of the first univariate equation $f_1=0$ which are dependent on $Y_k|_\Omega$.  Since
$Y_k|_\Omega$ is bounded by $C_1$, we see $x_1$ is bounded in terms of $C_1$.  
Then $x_2$ can be bounded from  the second equation which is now univariate if assuming $x_1$ 
is known. $x_2$ can be bounded in terms of the coefficients of $f_2$ and the bound on $x_1$. 
And so on. In summary, all the entries of $Y_k$ with indices in $\Omega^c$ can be bounded 
in terms of the entries in $Y_k|_{\Omega}$.  In other words, $\|Y_k\|\le 
C_2<\infty$ with a positive constant $C_2$ for all $k\ge 1$ which is dependent on $C_1$ above. 

It now follows that there exists a subsequence $Y_{k_j}$ which converges to $Y^\star$.  
Next by (\ref{fact3.2}), $\pp k$ are bounded because of $\qq k$ are bounded and 
hence, $\pp k,, k\ge 1$ have a convergent subsequence
and $\pp {k_j} \to X^\star$ when $k_j\to \infty$ without loss of generality. 
By Lemma~\ref{fact7}, we have $(Y^\star)_{\Omega^c}=(X^\star)_{\Omega^c}$. 
Finally, it is easy to see (\ref{mainresult2}) which follows from the facts that set $\mathcal{A}_\Omega$ and 
set $\overline{{\cal M}_{r_g}}$ are closed sets.	These complete the proof.   
\end{proof}

Although we do not know  how to check if there are only finitely many
matrices $Y\in\overline{{\cal M}_r}$ satisfying $(Y)_\Omega= {\bf x}$, we can see if the norms of $Y_k$ 
are bounded or not from the algorithm. If they are bounded, the conclusions of Theorem~\ref{fact8} hold. 
In general,  $X^\star \not= Y^\star$ as $r_g$ is not equal to $\hbox{rank}(M)$. For example, when $r_g<
\hbox{rank}(M)$, $Y^*$ will not be equal to $M$  and hence, $Y^\star$ does not satisfy $(Y^\star)_\Omega
=M_\Omega$ in general. Of course 
 $X^*$ satisfies the interpolation conditions $(X^*)_\Omega =M_\Omega$, but 
$\hbox{rank}(X^\star)$ may be bigger than $r_g$. That is, informally speaking, when $r_g<\hbox{rank}(M)$, the 
chance of $X^\star=M$ 
is bigger than the chance $Y^*=M$.  On the other hand, when $r_g>\hbox{rank}(M)$, 
there are more possibilities of matrices with rank $=r_g$ satisfying the interpolatory conditions. 
Anyway, if $X^\star - Y^\star\not=0$, the guess $r_g$ is not correct and we need to increase $r_g$. 

Finally, even though $X^*\not= Y^*$ in general, they satisfy the following nice property. 
\begin{proposition}
	Let $X^\star$ and $Y^\star$ be matrices in (\ref{mainresult2}) Then,
	$$
	Y^\star (X^\star)^\top = Y^\star (Y^\star)^\top
\hbox{ 	and } 
	(Y^\star)^\top X^\star = (Y^\star)^\top Y^\star. 
	$$
\end{proposition}
\begin{proof}
Using Lemma~\ref{fact11}, we obtain
	$$
	\langle AY^\star + Y^\star B, X^\star-Y^\star \rangle = 0
	$$ 
for all $A,B \in \mathbb{R}^{n \times n}$	which implies
$$
\langle A^\top, Y^\star(X^\star-Y^\star)^\top \rangle +  \langle B, 
(Y^\star)^\top(X^\star-Y^\star) \rangle = 0
$$ 
for all $A, B \in \mathbb{R}^{n \times n}$. 	Hence, 
	$$
	Y^\star(X^\star-Y^\star)^\top = 0
\hbox{ 	and } 
	(Y^\star)^\top (X^\star-Y^\star) = 0.
	$$ 
Rearranging the above equations, we obtain the required result. 
\end{proof}

\section{Numerical Results}
In this section, we first present some results based on the simple initial guess $X_0=P_\Omega(M)$. The 
robustness of Algorithm~\ref{alg1} was demonstrated in \cite{JZLS17}. We shall not repeat the similar numerical 
experimental results. We mainly present numerical results based on a good strategy to choose quality 
initial guesses which lead even better performance of Algorithm~\ref{alg1}. 
That is, we recall an efficient computational algorithm called OR1MP for matrix completion in \cite{WL15}. 
We use the OR1MP algorithm
to get a completed matrix which serves as an initial guess $X_0$. Our numerical experimental results show that 
this new initial guess gives more accurate completion. We measure the error matrices by using the maximum 
norm of all entries of the matrices. 
One can see that the maximum norm error is very small and hence, the recovered matrix is  very accurate. 
We shall also use Algorithm~\ref{alg1} to recover images from their partial pixel values and demonstrate that
Algorithm~\ref{alg1} is able to recover the images better visually.  Thus, this section is divided into two
subsections.

\subsection{Numerical Results: Initial Matrices from the OR1MP Algorithm}
In all the  experiments in this subsection, we used the initial matrix $X_0$ from the OR1MP algorithm 
in \cite{WL15} based on the $P_\Omega(M)$ using a few iterations, 
that is, $X_0=\hbox{OR1MP}(P_\Omega(M))$. 

\begin{example} In this example, we show the maximum missing rate that Algorithm~\ref{alg1} can recover a
matrix when its rank is fixed. Together we show the computational times. 
Abbreviations used in Tables in this example are as follows:\\
$\textbf{M.R.} = \text{Missing Rate, the fraction of missing entries} = \frac m {n^2}$,\\
$\textbf{O.R.} = \text{oversampling ratio} = \frac m {2nr-r^2}$,\\
$\textbf{M.C.E.} = \text{Maximum Component Error} = \max_{i,j} \abs{(X_{recovered})_{i,j}-M_{i,j}}$,\\
$\textbf{A.R.E.} = \text{Average Relative Error} = \norm{P_\Omega(Y_k)-P_\Omega(M)}_F
/\norm{P_\Omega(M)}_F$,\\
\begin{table}[!htbp]
	\caption{Numerical results based on $100 \times 100$ matrices averaged over 20 runs} 
	\centering 
	\begin{tabular}{c l c c c c c c} 
		\hline\hline 
		Rank & M.R.  & O.R. & M.C.E & A.R.E & Time  \\ 
		\hline 
		2 & 0.80 & 5 & 9.5202e-04 & 3.7217e-05 & 0.4171  \\ 
		5 & 0.61 & 4 & 6.0350e-04 & 1.0894e-05 & 0.2648 \\
		10 & 0.43 & 3 &  4.4343e-04 & 4.4977e-06 & 0.2778 \\
		20 & 0.28 & 2 & 6.0317e-04 & 2.0486e-06 & 0.8492\\
		35 & 0.25 & 1.3 & 1.2698e-06 & 0.0015 & 2.8798\\
		50 & 0.025 & 1.3 & 0.0013 & 7.2350e-07 & 1.2605\\ [1ex] 
		\hline 
	\end{tabular}
	\label{100cross100table} 
\end{table}

\begin{table}[!htbp]
	\caption{Numerical results based on  $250 \times 250$ matrices averaged over 20 runs} 
	\centering 
	\begin{tabular}{c c c c c c c c} 
		\hline\hline 
		Rank & M.R.  & O.R. & M.C.E & A.R.E & Time  \\  
		\hline 
		10 & 0.76 & 3 & 6.6930e-04 & 3.0595e-06 & 1.2283 \\ 
		20 & 0.53 & 3 & 2.2215e-04 & 1.0460e-06 & 1.3495 \\
		50 & 0.28 & 2 &  2.0560e-04 & 3.3083e-07 & 2.2624 \\
		75 & 0.18 & 1.6 & 2.6951e-04 & 2.0955e-07 & 4.5208 \\
		100 & 0.168 & 1.3 & 3.9345e-04 & 1.6690e-07 & 14.3622 \\
		125 & 0.025 & 1.3 & 6.2102e-04 & 1.1374e-07 & 8.8464 \\ [1ex] 
		\hline 
	\end{tabular}
	\label{250cross250table} 
\end{table}

\begin{table}[!htbp]
	\caption{Numerical results based on $500 \times 500$ matrices averaged over 10 runs} 
	\centering 
	\begin{tabular}{c c c c c c} 
		\hline\hline 
		Rank & M.R.  & O.R. & M.C.E & A.R.E & Time  \\  
		\hline 
		25 & 0.70 & 3 & 2.8565e-04 & 5.5169e-07 & 4.5253\\ 
		50 & 0.62 & 2 & 1.6818e-04 & 2.4458e-07 & 10.7270 \\
		100 & 0.28 & 2 &  8.6199e-05 & 8.0210e-08 & 11.3425 \\
		150 & 0.23 & 1.5 & 1.2031e-04 & 5.6053e-08 & 35.3097\\
		200 & 0.04 & 1.5 & 1.5896e-04 & 3.2623e-08 & 24.1821 \\
		250 & 0.0250 & 1.3 & 3.3090e-04 & 2.8449e-08 & 46.9267\\ [1ex] 
		\hline 
	\end{tabular}
	\label{500cross500table} 
\end{table}

\begin{table}[!htbp]
	\caption{Numerical results based on $1000 \times 1000$ matrices averaged over 10 runs} 
	\centering 
	\begin{tabular}{c c c c c c c c} 
		\hline\hline 
		Rank & M.R.  & O.R. & M.C.E & A.R.E & Time  \\ 
		\hline 
		50 & 0.70 & 3 & 6.8718e-05 & 1.3722e-07 & 30.1813  \\ 
		100 & 0.52 & 2.5 & 3.7074e-05 & 5.2213e-08 & 50.0631 \\
		200 & 0.10 & 2.5 &  2.6120e-05 & 1.2338e-08 & 42.7043 \\
		300 & 0.05 & 1.85 & 5.1339e-05 & 1.0448e-08 & 83.9782 \\
		400 & 0.04 & 1.5 & 7.1099e-05 & 8.0391e-09 & 186.2271 \\
		500 & 0.0025 & 1.33 & 2.4592e-04 & 6.5708e-09 & 226.3912 \\ [1ex] 
		\hline 
	\end{tabular}
	\label{1000cross1000table} 
\end{table}
\end{example}

\begin{example}
Next we provide another tables to show that our algorithm is very 
effective in recovering the original matrix. We let the missing rate 
$=0.1, 0.2, \cdots, 0.9$ and find the largest rank our algorithm can 
complete within maximum norm error $<1e-3$, that is, every entry of
the completed matrix is
accurate to the first three digits. That is, for a fixed missing rate $\delta$, we randomly find the known indices set $\Omega$ with 
$|\Omega|/(n^2)=1-\delta$ and then we randomly generate a matrix 
$M$ of size $n\times n$ with rank $r\ge 1$. We   
use $M_\Omega$, $\Omega$, and $r$ to recover $M$ (the stopping 
criterion is $1e-5$ of the consecutive iterations),  check if 
the completed matrix $\widehat{M}$ approximates $M$ in the maximum norm 
within $\epsilon =1e-3$, and repeat the computation in 10 times.
If all 10 computations are able to accurately recover $M$,
 we advance $r$ by $r+1$ and repeat the above procedures until the 
accurate recovery is less than 10 times for a fixed $r$. 
In this way, we can 
find the largest rank for a fixed missing rate. As we  use two initial 
guesses, we summarize the computational results in 
Table~\ref{maxranktable}.  

\begin{table}[!htbp]
\centering
\begin{tabular}{|c|c|r|r|r|r|r|r|r|r|r|} 
\hline 
missing rates &  0.1 & 0.2 & 0.3 & 0.4 & 0.5 & 0.6 & 0.7 & 0.8 & 0.9 &  \cr\hline
largest ranks &   30 & 16 &   19 &   14 &    9 &    7&    5 &    2 &  1 & OR1MP  \cr \hline
largest ranks &   13&  14 &   13 &   10&    9 &    7&    3&    2 &    1 & $M_\Omega$\cr \hline 
\end{tabular}
\caption{maximum ranks are based on matrices of size $100\times 100$ with initial values from OR1MP (second row) and 
from the initial matrix $M_\Omega$ (third row) \label{maxranktable}} 
\end{table}

From Table~\ref{maxranktable}, we can see that using OR1MP algorithm to generate an initial guess 
for Algorithm~\ref{alg1} 
is much better when the rates of missing entries are small. When the rate 
of missing entries are large, the performance is similar.  
If this table is compared with the ones in \cite{WCCL16}, we remind the reader that 
we use a much tougher criterion $\epsilon=1e-3$ in the maximum norm to 
find the maximum rank than the relative Frobenius norm error used in \cite{WCCL16}. 

If we use the standard relative Frobenius norm error, we have largest ranks that Algorithm~\ref{alg1} 
can recover 100\% times listed in Table~\ref{relativeerror} with two different initial guesses. 
We can see that the performance increases greatly when using a completed matrix from OR1MP algorithm. 
\begin{table}[!htbp]
\centering
\begin{tabular}{|c|c|r|r|r|r|r|r|r|r|r|c|} 
\hline 
missing rates & 0.1 & 0.2 & 0.3 & 0.4 & 0.5 & 0.6 & 0.7 & 0.8 & 0.9 & \cr\hline 
largest ranks & 132 &  106 &   83 &  68 &   41  &   40 &   25 &   14 &    4 & OR1MP \cr \hline 
largest ranks &   33 &  29 &  25  &  24 &   19  &   15 &   11 &    8 &    4 & $M_\Omega$\cr \hline 
\end{tabular}
\caption{maximum ranks are based on matrices of size $200\times 200$ with initial values from OR1MC (second row)
and from the initial matrix $M_\Omega$ (third row)\label{relativeerror}}
\end{table} 
\end{example}

\subsection{Image Recovery from Partial Pixel Values}
We shall use Algorithm~\ref{alg1} to recover images from partial 
pixel values. 

\begin{example} 
Let us use the standard  images knee, penny and thank as testing matrices  
of pixel values. The image knee is of size $691\times 691$. 
The image penny is a matrix of size $128\times 128$ and the image thank 
is of size $300\times 300$. For image knee, 
we use a missing rate $0.85$ to generate $M_\Omega$  and use rank=25 to find 
an approximation of the  image knee by using the well-known matrix completion OR1MP algorithm in \cite{WL15},
then we feed the approximation as an initial guess to Algorithm~\ref{alg1} to get a better approximation. 
Also we use the same known entries $M_\Omega$ as an initial guess in  our Algorithm~\ref{alg1} 
to find an approximation of the image directly.   All these images are shown  in Figure~\ref{knee}. 
We do the same for the  images penny and thank. See  Figures~\ref{penny} and ~\ref{thanks}.  
Visually, we can see that starting from an initial guess obtained from the OR1MP algorithm, 
our Algorithm~\ref{alg1} produces a much better approximation to the image. For image
penny, we are able to see the face of Lincoln and the word as well as 
number 1984 are much cleaner although the root-mean square error (RMSE) may not be better.  
Many images have been experimented with similar performance. 
\begin{figure}[htpb]
\centering
\begin{tabular}{ccc}
\includegraphics[scale=0.18]{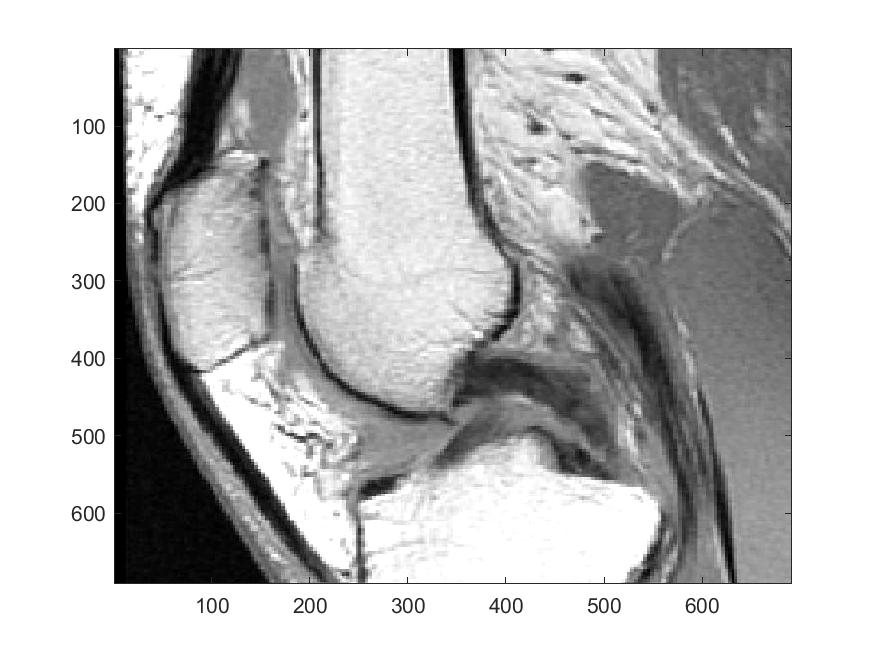} & \includegraphics[scale=0.18]{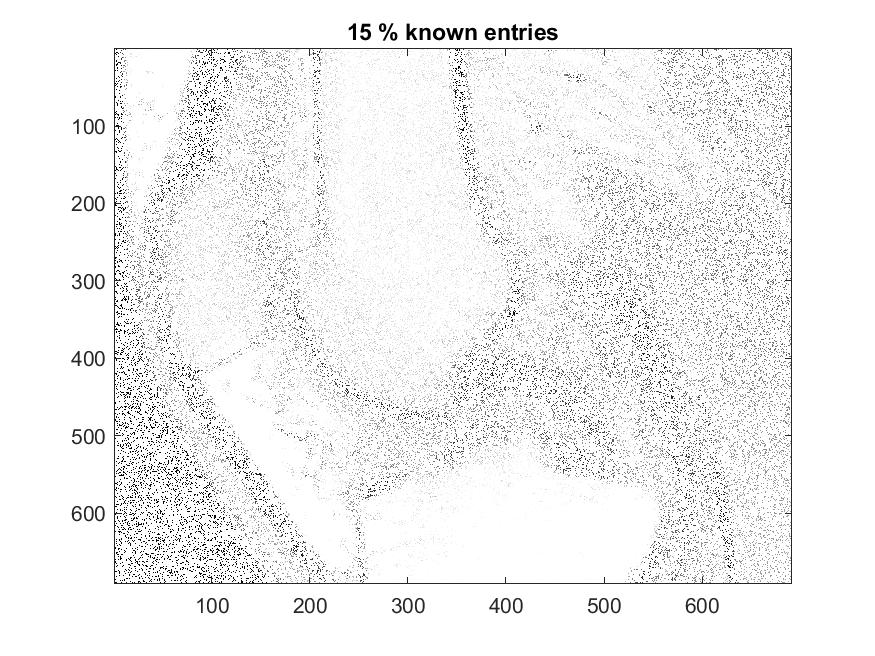} & \cr
\includegraphics[scale=0.18]{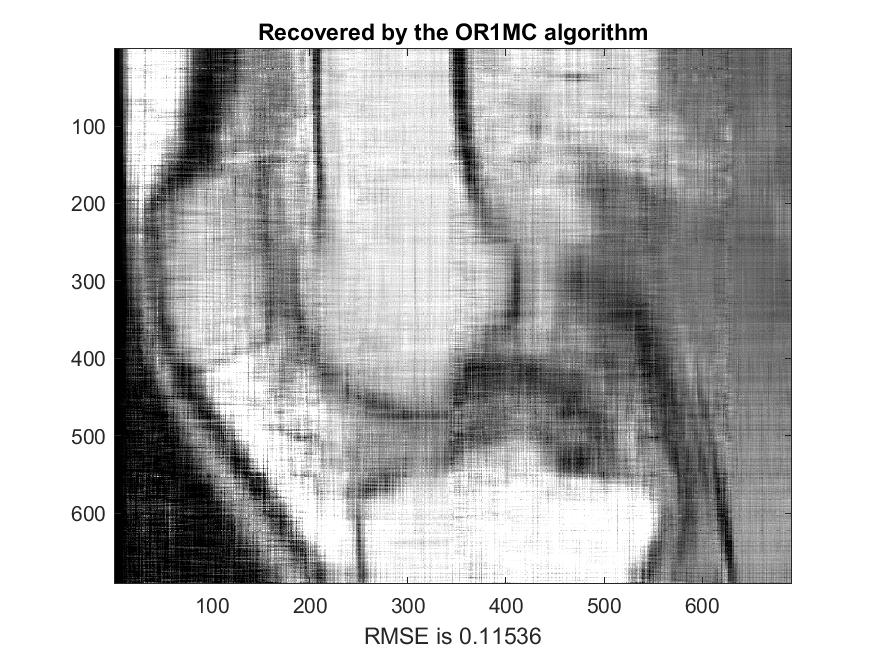} & \includegraphics[scale=0.18]{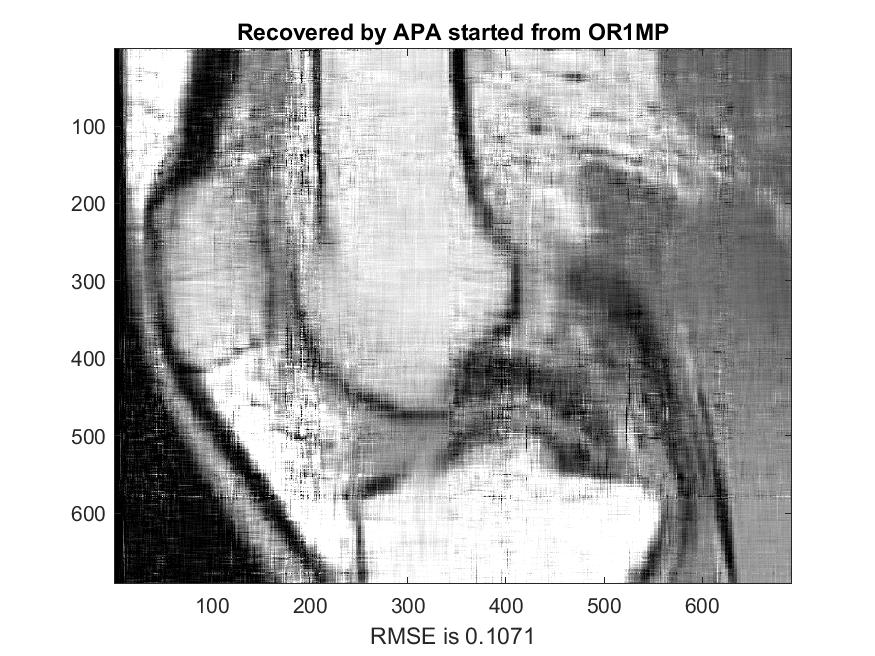}
& \includegraphics[scale=0.18]{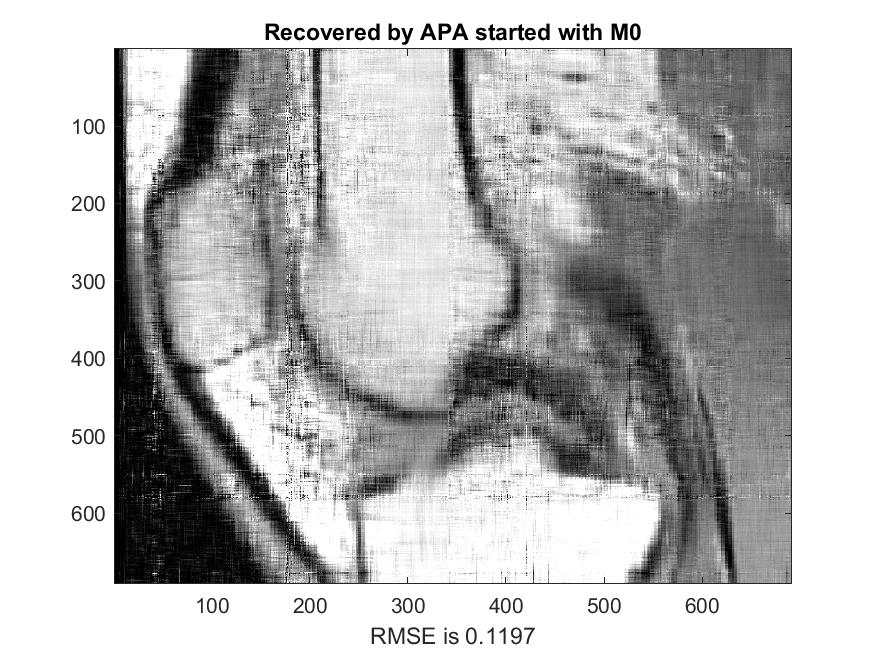}\cr
\end{tabular}
\caption{The top row: The original image and the image of 15\% known entries; The bottom row: 
The outputs from Algorithm OR1MP, Algorithm~\ref{alg1} with initial guess from the Algorithm OR1MP and
Algorthm~\ref{alg1} from the 15\% known entries based on rank 25. \label{knee}}		
\end{figure}


\begin{figure}[htpb]
\centering
\begin{tabular}{ccc}
\includegraphics[scale=0.18]{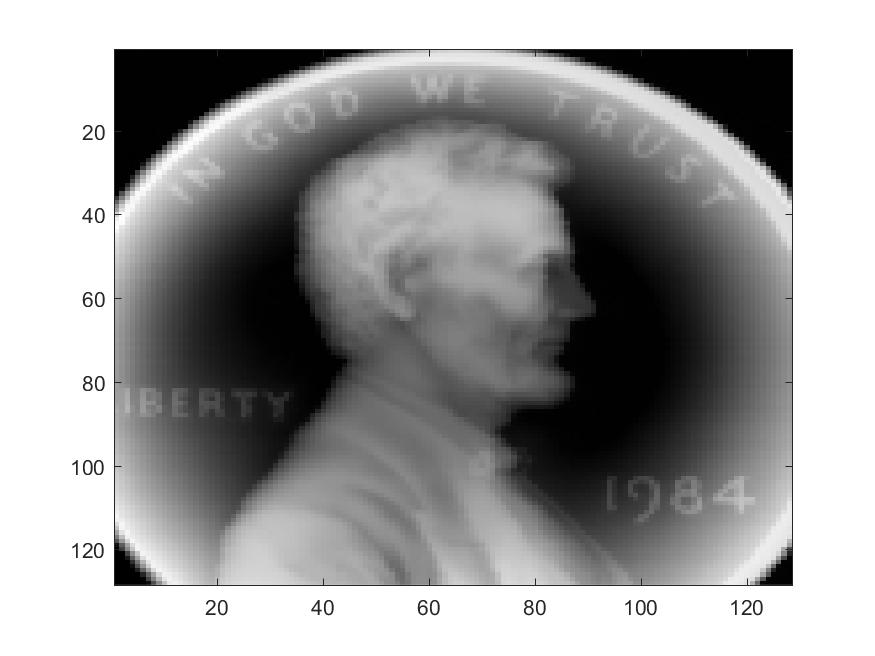} & \includegraphics[scale=0.18]{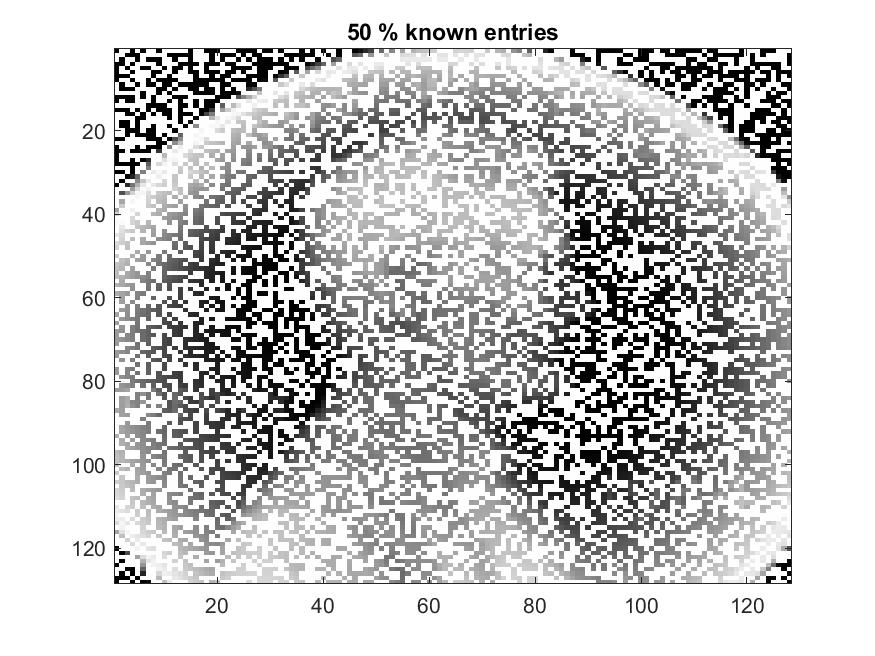}&\cr
\includegraphics[scale=0.18]{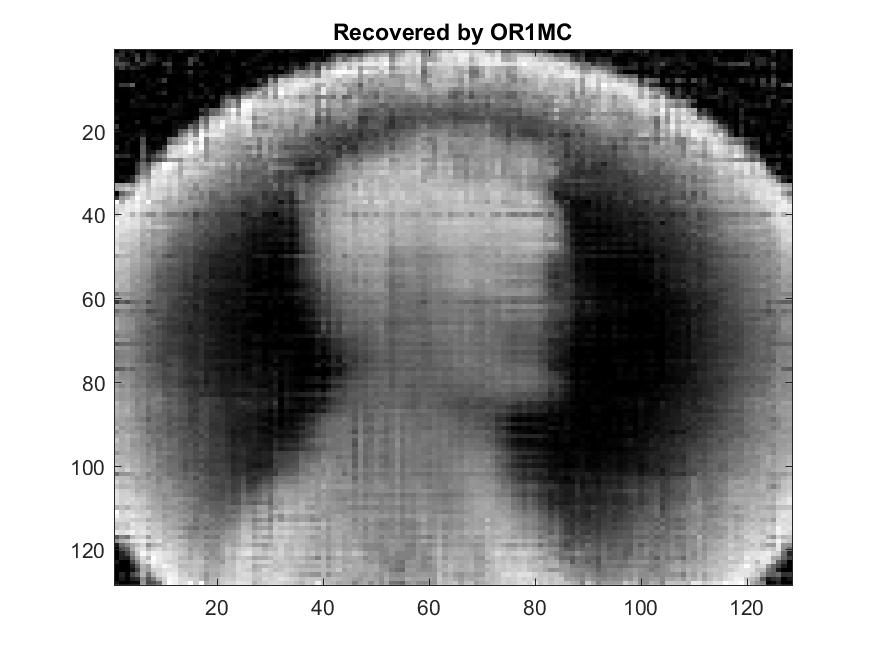} & \includegraphics[scale=0.18]{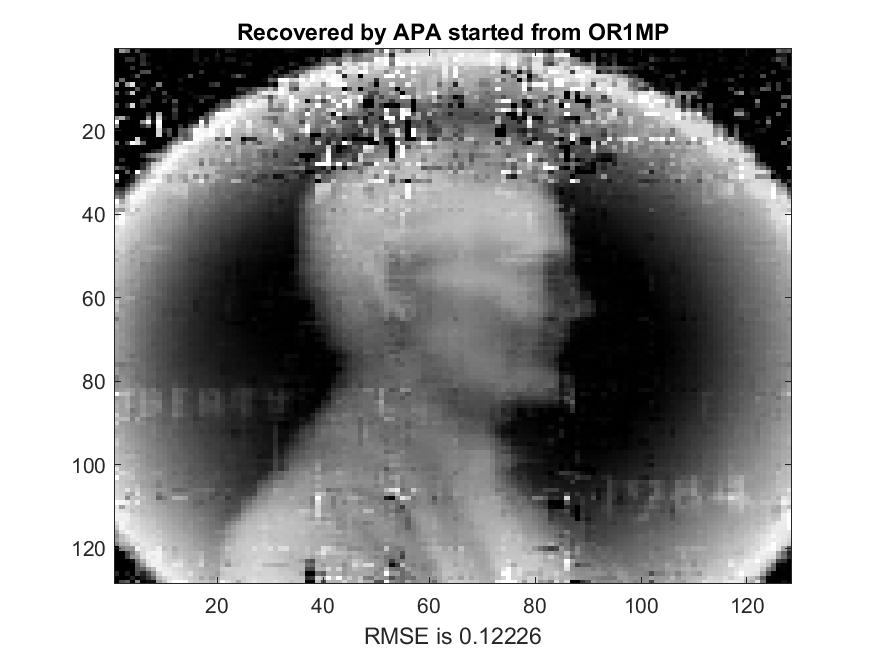}
& \includegraphics[scale=0.18]{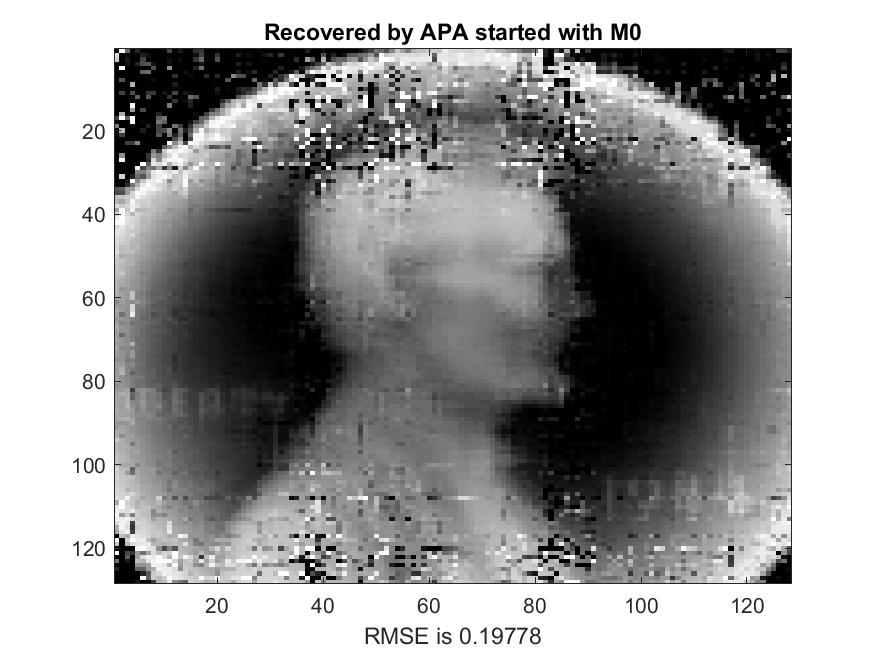}\cr
\end{tabular}
\caption{The top row: The original image and the image of 50\% known entries; The bottom row: 
The outputs from Algorithm OR1MP, Algorithm~\ref{alg1} with initial guess from the Algorithm OR1MP and
Algorthm~\ref{alg1} from the 50\% known entries based on rank 25.  \label{penny}}		
\end{figure}  

\begin{figure}[htpb]
\centering
\begin{tabular}{ccc}
\includegraphics[scale=0.18]{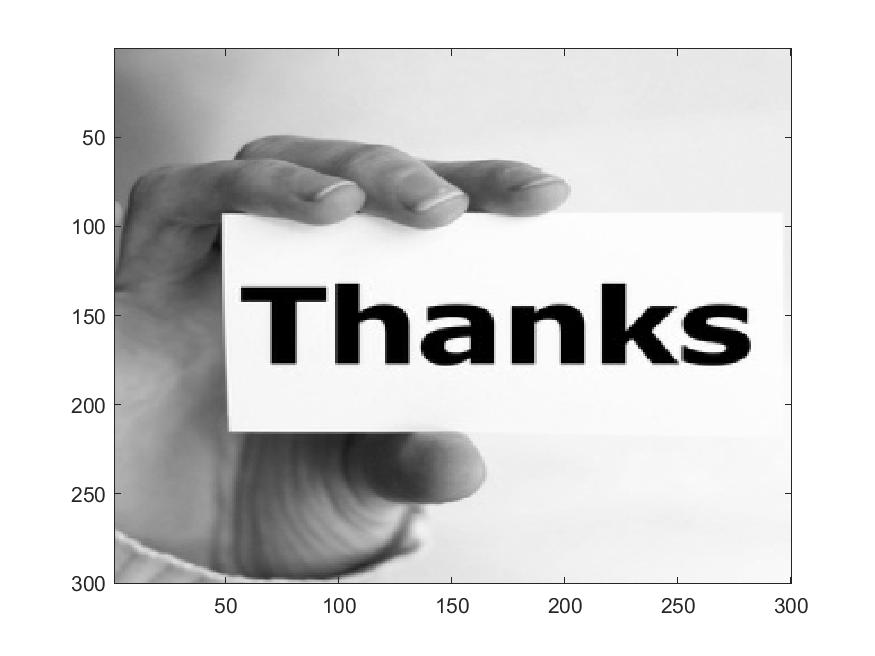} & \includegraphics[scale=0.18]{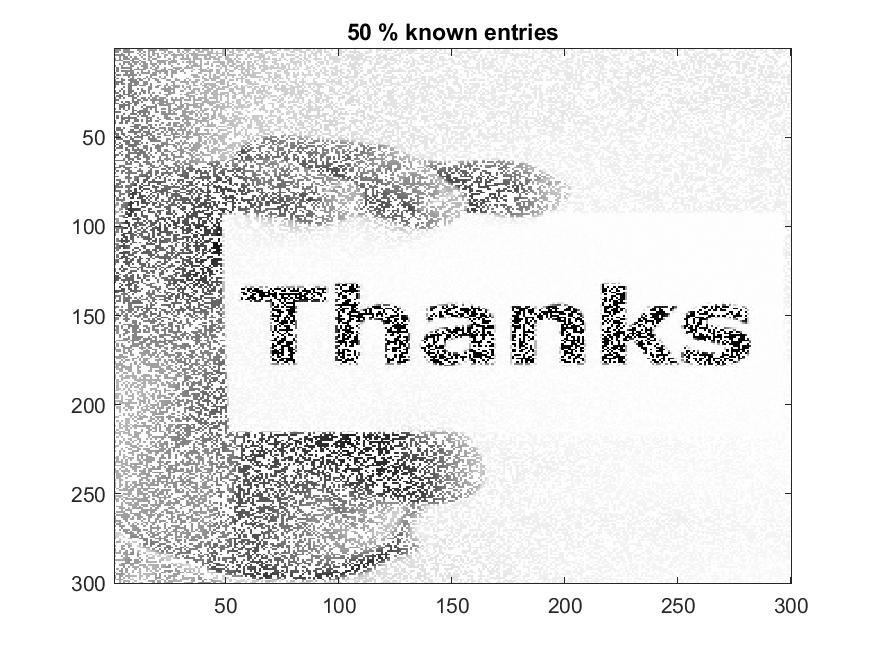} &\cr
\includegraphics[scale=0.18]{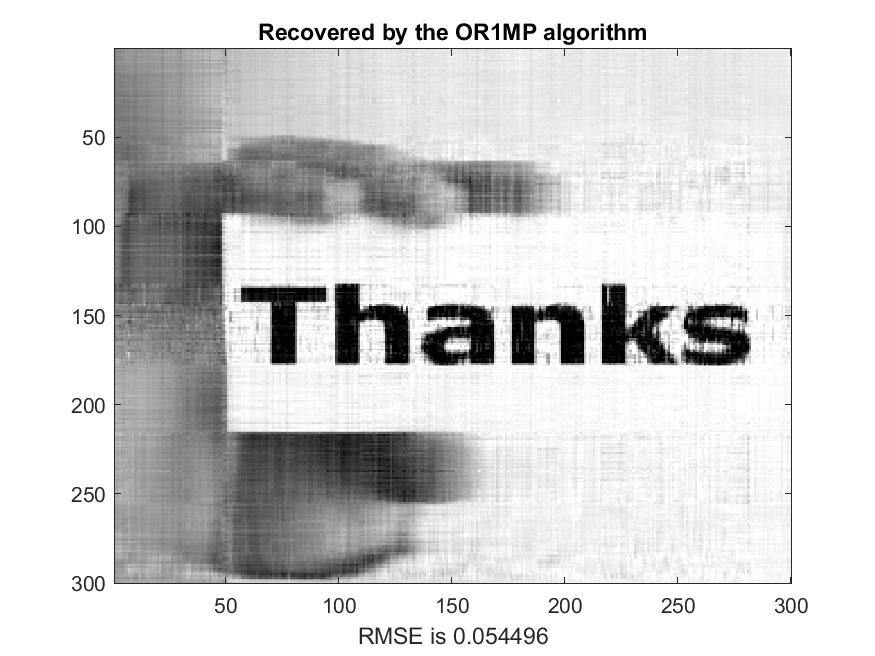} & \includegraphics[scale=0.18]{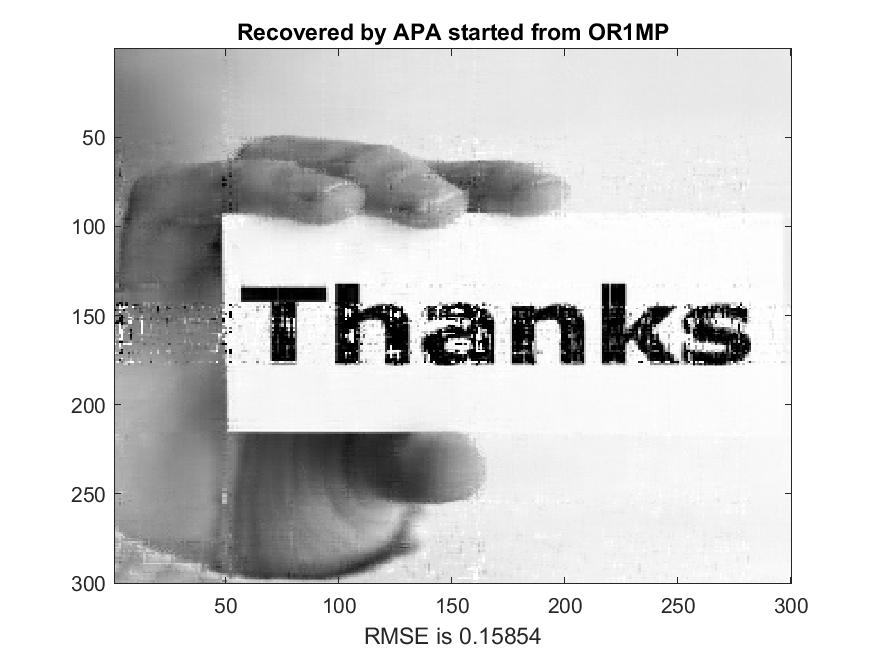}& 
\includegraphics[scale=0.18]{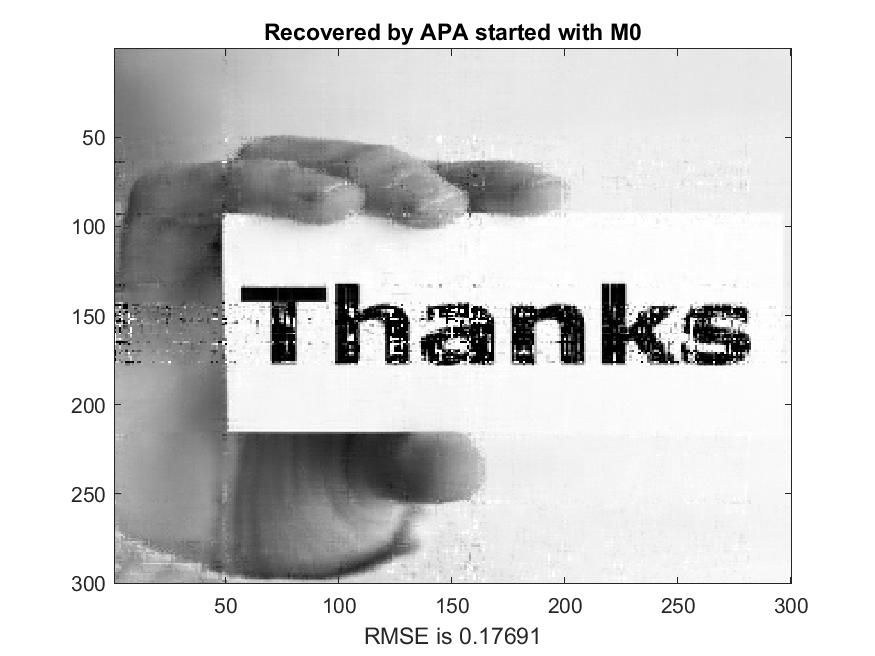}\cr
\end{tabular}
\caption{The top row: The original image and the image of 50\% known entries; The bottom row: 
The outputs from Algorithm OR1MP, Algorithm~\ref{alg1} with initial guess from the Algorithm OR1MP and
Algorthm~\ref{alg1} from the 50\% known entries based on rank 25.  \label{thanks}}		
\end{figure}  
\end{example}

\section{Alternating Projection Algorithm for Sparse Solution Recovery Problem}
In this section, we will use the same ideas of alternating projection discussed in the previous section  
to study the following classical problem in the area of compressed sensing:

\begin{align}\label{CS}
	&\underset{X}{\operatorname{minimize}}\quad \norm{X}_0 \\
	&\operatorname{subject\;to} \quad 
	 A {\bf x} = {\bf b},
\end{align}	
where $A \in \mathbb{R}^{n \times N}, {\bf x}\in \mathbb{R}^N, {\bf b} \in \mathbb{R}^{n}, n << N$ and 
$\norm{ {\bf x} }_0$ is the $\ell_0$ quasi-norm 
of a vector ${\bf x}$. Recall that the  $\ell_0$ quasi-norm of a vector is the number of non-zero components of the vector.
Let $\mathcal{L}_s(\mathbb{R}^N)$ denote the collection of all $s-$sparse vectors in $\mathbb{R}^N$,
$$
\mathcal{L}_s(\mathbb{R}^N) := \left\{x \in \mathbb{R}^N \mid \quad \norm{x}_0 = s \right\}
$$
and $\mathcal{P}_{\mathcal{L}_s}$ and $\mathcal{P}_{\mathcal{A}}$ denote the projection onto the set $\mathcal{L}_s(\mathbb{R}^N)$ 
and the affine space $\mathcal{A}:=\{{\bf x}:  A {\bf x} ={\bf b}\}$, respectively. It is easy to know $\mathcal{A}= \hbox{Null}(A)+
{\bf x}_0$, where ${\bf x}_0\in \mathbb{R}^N$ satisfies $A{\bf x}_0= {\bf b}$. 
Note that the projection $\mathcal{P}_{\mathcal{L}_s}({\bf x}_k)$ can be computed easily 
by setting the smallest $n-s$ components of the vector  ${\bf x}_k$ to zero. 

Our algorithm can be stated as follows:
 
\begin{algorithm}[H]
\label{alg2}
	\KwData{Sparsity s of the solution ${\bf x}_\star$, the tolerance $\epsilon$ whose default value is 1e-6}
	\KwResult{${\bf x}_k$ a close approximation of ${\bf x}_\star$}
	Initialize ${\bf x}_0$ to a random vector in the affine space $\mathcal{A}$\;
	\Repeat{The smallest $n_2-s$ components of ${\bf x}_{k+1}$ have magnitude less than $\epsilon$}{
		\textbf{Step 1:} ${\bf y}_k = \mathcal{P}_{\mathcal{L}_s}({\bf x}_k)$\\
		\textbf{Step 2:} ${\bf x}_{k+1} = \mathcal{P}_{\mathcal{A}}({\bf y}_k)$\;
	}
\caption{Alternating Projection Algorithm for $\ell_0$ Minimization}
\end{algorithm}

We first discuss the convergence of Algorithm~\ref{alg2}. Then we shall present its numerical performance in the next section. As a good 
initial guess is very important to have a quick convergence, we shall explain a few approaches to obtain reasonable initial guesses. 

\subsection{Convergence of Algorithm~\ref{alg2}}
We begin with some elementary results.
\begin{lemma}
Let	$\mathcal{L}_s(\mathbb{R}^n)$ be the collection defined as follows. 
$$
	\mathcal{L}_s(\mathbb{R}^n) = \mathop{\bigsqcup}_{\mathcal{I}}  \{x \in \mathbb{R}
	^n \mid x_j = 0 \quad \forall j \in \mathcal{I}^c  \},
$$
where the index set $\mathcal{I}$ ranges over all the subsets of $\{1,2,\cdots,n_1\}$ which 
has cardinality $s$. Here, $\mathop{\bigsqcup}_{\mathcal{I}}$ stands for the disjoint union over 
$\mathcal{I}$. Then $\mathcal{L}_s(\mathbb{R}^n)$ consists of a disjoint union of affine spaces. 
\end{lemma}
\begin{proof}
It is easy to see that the statement is correct. \end{proof}
 
\begin{lemma}
The set of vectors in $\mathbb{R}^{N}$ for which $\mathcal{P}_{\mathcal{L}_s}(x)$ is single-valued, is given
by the open set
$$V_s = \left\{x \in \mathbb{R}^{n_2} \mid \abs{x_{i_1}} \geq  \abs{x_{i_2}} \geq \cdots \abs{x_{i_{n_2}}}, \abs{x_{i_{s+1}}} \neq \abs{x_{i_s}} \right\}
$$
consisting of vectors which has the property that if one arrange the components in decreasing order of magnitude, 
then $s^{th}$ and $(s+1)^{th}$ terms are distinct.
\end{lemma}
\begin{proof}
We first start by noting that the projection $\mathcal{P}_{\mathcal{L}_s}({\bf x})$ is obtained by setting
the smallest $N-s$ components in magnitude of the vector $x$ to zero.  Hence, the projection is single-valued if the $N-s$ 
smallest components of ${\bf x}$ are in unique positions(indices). Hence we must have that the $(N-s)^{th}$ and $(N-s + 1)^{th}$ 
components of $x$ must be distinct. Now we will show that the set $V_s$ is an open set.
Let ${\bf x} = (x_1, x_2, \cdots, x_{N}) \in \mathbb{R}^{N}$ with $\abs{x_{i_1}} \geq  \abs{x_{i_2}} \geq \cdots \abs{x_{i_{N}}}, 
\abs{x_{i_{s+1}}} \neq \abs{x_{i_s}}$. Let 
$$
	\epsilon := \frac{\abs{\abs{x_{i_{s+1}}} - \abs{x_{i_s}}}}{4}
$$
Consider an open ball $B_\epsilon({\bf x})$ centered at ${\bf x}$ of radius $\epsilon$. 
We have, for all ${\bf y} \in  B_\epsilon({\bf x})$ and $j \in \{1, 2, \cdots, N\}$, 
$$
\abs{\abs{y_j}-\abs{x_j}} \leq \abs{y_j-x_j} \leq \norm{{\bf y} -{\bf x}} < \epsilon
$$
Therefore, we have 
$$
\abs{y_{i_{j+1}}} \leq \abs{x_{i_{j+1}}} + \abs{\abs{y_{i_{j+1}}}-\abs{x_{i_{j+1}}}} < \abs{x_{i_{s+1}}} + \epsilon < 
\frac{\abs{\abs{x_{i_{s+1}}} + \abs{x_{i_{s}}}}}{2}
$$
for $j \geq s$. Similarly,
$$
\abs{y_{i_{j}}} \geq \abs{x_{i_{j}}} - \abs{\abs{x_{i_{j}}}-\abs{y_{i_{j}}}} > \abs{x_{i_{s}}} - \epsilon > \frac{\abs{\abs{x_{i_{s+1}}} + 
\abs{x_{i_{s}}}}}{2}
$$
for $j\leq s$. Hence, we deduce that, for all ${\bf y} \in B_\epsilon({\bf x})$ and 
$j \in \{1,2,\cdots, s\}$, $\abs{y_{i_{j}}} > \abs{y_{i_s}}$, which implies that ${\bf y} \in V_s$ and, therefore, 
$B_\epsilon({\bf x}) \subset V_s$. 
\end{proof} 

Next let us recall the following well-known results. 
\begin{theorem}[Von Neumann, 1950\cite{Neumann}]
\label{VonN}
If $L_1$ and $L_2$ are two closed subspaces of a Hilbert space $X$, then the sequence of operators
$$
\mathcal{P}_{L_1}, \mathcal{P}_{L_2}\mathcal{P}_{L_1}, \mathcal{P}_{L_1}\mathcal{P}_{L_2}\mathcal{P}_{L_1}, 
\mathcal{P}_{L_2}\mathcal{P}_{L_1}\mathcal{P}_{L_2}\mathcal{P}_{L_1}, \cdots
$$
converge to $\mathcal{P}_{L_1\cap L_2}$. In other words,
$$
\lim\limits_{k \rightarrow \infty}(\mathcal{P}_{L_2}\mathcal{P}_{L_1})^k(x) = \mathcal{P}_{L_1\cap L_2}(x)
$$
for all $x \in X$.	
\end{theorem}
\begin{proof}
Refer to \cite{Neumann} Chapter 13, Theorem 13.7 for a proof. 
\end{proof}

\begin{theorem}
\label{mainresult3}
If ${\bf x}_\star$ is an isolated point of $\mathcal{L}_s(\mathbb{R}^N)\cap\mathcal{A}$. Then, Algorithm~\ref{alg2} 
will locally converge to ${\bf x}_\star$ linearly. 
\end{theorem}
\begin{proof}
Let $\mathcal{I} = \operatorname{Supp}({\bf x}_\star)$ be the support of ${\bf x}_\star$ and $s = \norm{{\bf x}_\star}_0$. 
Consider an open set $V_s$ of vectors which has the property that their $n-s$ smallest components are are in unique positions(indices). 
In fact, $V_s$ can be concretely described as 
$$
V_s = \left\{ {\bf x} \in \mathbb{R}^{N} \mid \quad \abs{x_{i_1}}\geq \abs{x_{i_2}} \geq  \cdots \abs{x_{i_{n_2}}},\abs{x_{i_{s}}} 
\neq \abs{x_{i_{s+1}}} \right\}.
$$
	
Clearly ${\bf x}_\star \in V_s$. Let $B(r)$ be an open ball centered at ${\bf x}_\star$ and of radius $r$ completely contained inside 
$V_s$. Since $B(r) \subseteq V_s$, for any ${\bf x} \in B(r)$, the projection  $\mathcal{P}_{\mathcal{L}_s}({\bf x})$ is uniquely defined. 
Since affine spaces in a finite dimensional Euclidean space are closed, one can shrink the ball $B(r)$, if necessary, such that the restriction 
$\mathcal{L}_s(\mathbb{R}^{n_2})\vert_{B(r)}$ of the set of $s-$sparse vectors to the open set $B(r)$ is an affine space. 
Then under the assumption the hypothesis in this theorem, the result follows from Theorem \ref{VonN}. 
\end{proof}

\begin{lemma}
\label{newkey}
Assume $A$ has the following property: 
\begin{equation}
\label{critical}
\mathcal{L}_s(\mathbb{R}^{N}) \cap \operatorname{Null}(A) = \{0\},
\end{equation}
where $\operatorname{Null}(A) $ is the null space of $A$. 
Furthermore, assume that ${\bf x}_\star \in \mathcal{L}_s(\mathbb{R}^{N})\cap \mathcal{A}$. Then  ${\bf x}_\star$ is an isolated point of 
$\mathcal{L}_s(\mathbb{R}^{N})\cap\mathcal{A}$. 
\end{lemma}
\begin{proof}
Assume, on the contrary, that ${\bf x}_\star$ is not an isolated point of the set $\mathcal{L}_s(\mathbb{R}^{n_2})\cap\mathcal{A}$. Then, since 
$A$ and $\mathcal{L}_s(\mathbb{R}^{N})$ are locally affine spaces, there exist a linear space $L$ of dimension greater than or equal to 1 such 
that $L+{\bf x}_{\star} \subseteq \mathcal{L}_s(\mathbb{R}^{N})\cap\mathcal{A}$. Since each
of the intersecting spaces are affine spaces locally, $L$ must lie also in the intersection	of their tangent spaces.   
Hence,
$$
L \subseteq T_{\mathcal{L}_s(\mathbb{R}^{N})}({\bf x}_\star)\cap \operatorname{Null}(A)
$$
where $T_{\mathcal{L}_s(\mathbb{R}^{N})}({\bf x}_\star)$ is the tangent space to $\mathcal{L}_s(\mathbb{R}^{N})$ at the point 
${\bf x}_\star$.
Now, since $\mathcal{L}_s(\mathbb{R}^{N})$ is an union of linear spaces, let us assume ${\bf x}_\star \in L_0 \subseteq 
\mathcal{L}_s(\mathbb{R}^{N})$ lies in a  linear space $L_0$ contained in $\mathcal{L}_s(\mathbb{R}^{N})$.
Therefore, we have
$$
	L \subseteq T_{L_0}({\bf x}_\star)\cap \operatorname{Null}(A) = L_0 \cap \operatorname{Null}(A) \subseteq  	
\mathcal{L}_s(\mathbb{R}^{N}) \cap \operatorname{Null}(A) =\{0\} 
$$
which leads to the  contradiction as $L$ is of dimension greater than or equal to $1$. 
Note that, in order to derive the equality in the last equation, 
we have used the fact that the tangent space of a linear space is the linear space itself.
\end{proof}

The discussion above leads to our final result in this section.
\begin{theorem}
Under the assumption (\ref{critical}) in Lemma~\ref{newkey}, Algorithm~\ref{alg2} will converge linearly for any starting initial guess
${\bf x}_0$. 
\end{theorem}
\begin{proof}
We simply combine Lemma~\ref{newkey} and Theorem~\ref{mainresult3} together to have this result. 
\end{proof}

\subsection{Numerical Results from Algorithm~\ref{alg2} for Sparse Vector Recovery}
We have used Algorithm~\ref{alg2} to compute sparse solutions and compare the performance of several existing 
algorithms. Mainly, we compare with the iteratively reweighted $\ell_1$ minimization (CWB for short) in 
\cite{CWB08}, the $L^1$ greedy algorithm (KP) proposed in \cite{KP10}, the FISTA in \cite{BT09}, 
the hard iterative pursuit (HTP) in \cite{F11}, and generalized approximate message passing algorithm (GAMP) 
in \cite{DMM09}, \cite{R11}. LV stands for our Algorithm~\ref{alg2}. 
We present the frequency of recovery of Gaussian random matrices of size $128\times 256$ with sparsity from 
$10--70$ over 500 repeated runs with a tolerance $1e-3$ in maximum norm. 
In Figure~\ref{CSgaussian}(left figure), we show the performance of various algorithms. 
Next we repeat the same experiments based on  uniform random matrices of size $128\times 256$. The
performance of frequency of recovery from various algorithm is shown in Figure~\ref{CSgaussian} (right).   
In this case, it is known that the GAMP is not good. 
\begin{figure}[h]
\centering
\begin{tabular}{cc}
\includegraphics[scale=0.2]{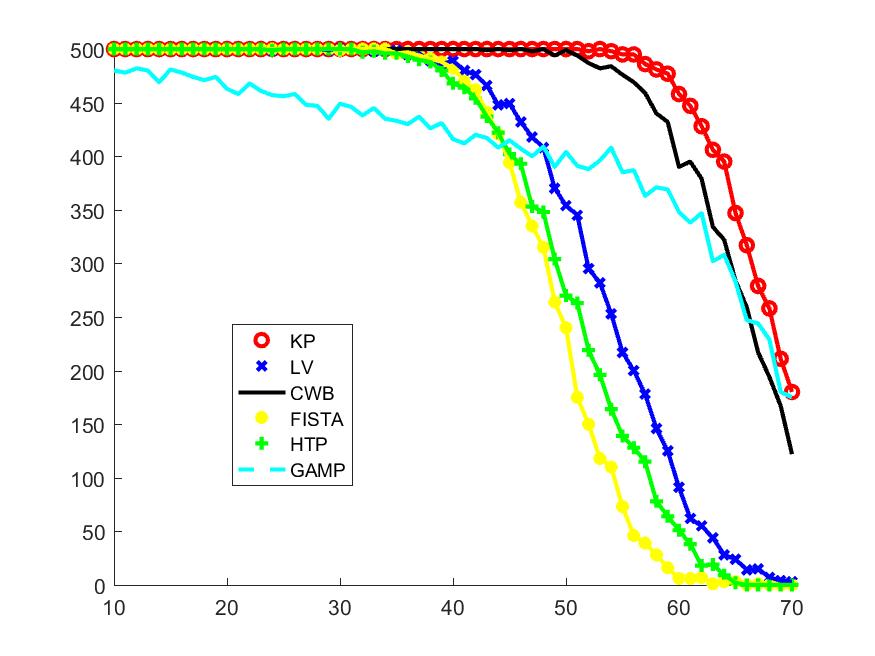} \includegraphics[scale=0.2]{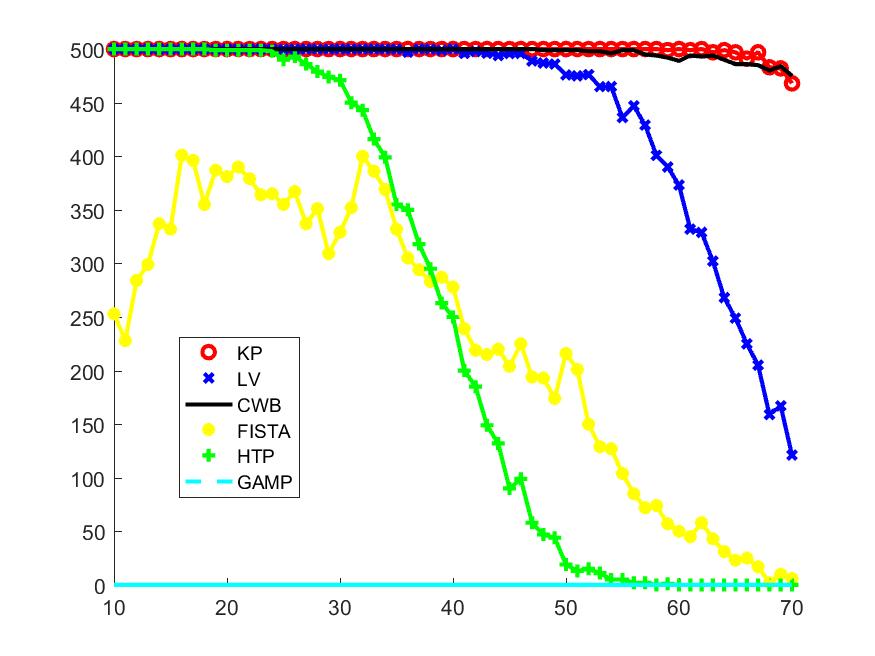}
\end{tabular}
\caption{Frequency of Sparse Recovery by Various Algorithms from Gaussian random matrices (left) 
and  from uniform random matrices(right) 
\label{CSgaussian}}		
\end{figure}

\section{Remarks on Existence of Matrix Completion}
Recall ${\cal M}_r$ is the set of all matrices of size $n\times n$ with rank $r$ and 
$\overline{\mathcal{M}_{r}}$ is the set of all matrices with rank 
$\le r$. It is clear that $\overline{\mathcal{M}_{r}}$ is the closure of ${\cal M}_r$ in the Zariski sense
(cf. \cite{ZARISKI}). 
It is easy to see that dimension ${\cal M}_r$ is $2 n r-r^2$ (cf. Proposition 12.2 in \cite{JoeHarris} for a proof).   Then  the dimension of
$\overline{\mathcal{M}}_r$ is also $2nr- r^2$. 
Also, it is clear that $\overline{\mathcal{M}_{r}}$ is an algebraic variety. 
In fact, $\overline{\mathcal{M}_{r}}$ is an irreducible variety. 
\begin{lemma}\label{irreduciblity}
$\overline{\mathcal{M}_{r}}$ is an irreducible variety.. 
\end{lemma}
\begin{proof}
Denote by $GL(n)$ the set of invertible $n\times n$ matrices. Consider the action of $GL(n)\times GL(n)$ 
on $M_n(R)$ given by: $ (G_1,G_2)\cdot M \mapsto G_1 M G_2^{-1}$, for all 
$G_1, G_2\in GL(n)$.  Fix a rank $r$ matrix $M$. Then the variety 
$\mathcal{M}_{r}$ is the orbit of $M$. 
Hence, we have a surjective morphism,  a regular algebraic map described by polynomials,  
from $GL(n)\times GL(n)$ onto $\mathcal{M}_{r}$. 
Since $GL(n)\times GL(n)$ is an irreducible variety, so is $\mathcal{M}_{r}$. 
Hence, the closure $\overline{\mathcal{M}_{r_g}}$ of the irreducible set $\mathcal{M}_{r_g}$ 
is also irreducible \emph{c.f} (cf. Example I.1.4 in \cite{Hartshorne}). 
\end{proof}

Consider the map
$$
\Phi_\Omega : \overline{\mathcal{M}_{r}} \rightarrow \mathbb{C}^{m}
$$ 
given by projecting any matrix $X \in \overline{\mathcal{M}_{r}}$ 
to its entries in position $\Omega$ which form a vector in $\mathbb{R}^{m}$.  Thus, 
$\Phi_\Omega(\overline{\mathcal{M}_r})$ are exactly the set of all $r-$feasible vectors in $\mathbb{C}^m$. 
As
the projection $\Phi_\Omega$ is nice (not like a Peano curve mapping $[0, 1] \to [0, 1]^2$), 
we expect that $\dim( 
\Phi_\Omega(\overline{\mathcal{M}_r}))$ is less than or equal to $\dim(\overline{\mathcal{M}_r})$ which 
is less than the dimension of $\mathbb{C}^m$. Thus, $\Phi_\Omega(\overline{\mathcal{M}_r})$ is not 
able to occupy the whole space $\mathbb{C}^m$. The Lebesgue measure of $\Phi_\Omega(\overline{\mathcal{M}_r})$
is zero and hence, randomly choosing a vector ${\bf x}\in \mathbb{C}^m$ will not be in 
$\Phi_\Omega(\overline{\mathcal{M}_r})$ most likely. Certainly, these intuitions should be made more precise. 
Recall the following result from Theorem 1.25 in Sec 6.3 of \cite{Shafarevich}.  
\begin{lemma}\label{fiberdimensionlemma}
Let $f : X \rightarrow Y$ be a regular map between irreducible varieties. 
Suppose that $f$ is surjective: $f(X) = Y$ , and that $\dim(X) = n$, $\dim(Y) = m$. Then $m \leq n$,
and 
\begin{enumerate}
	\item $\dim(F) \geq n-m$ for any $y \in Y$ and for any component $F$ of the fibre $f^{-1}(y)$;
	\item there exists a nonempty open subset $U \subset Y$ such that $\dim(f^{-1}(y)) = n - m$ for $y \in U$.
\end{enumerate}
\end{lemma}

We are now ready to prove 
\begin{theorem}
\label{nochance}
If one chooses randomly the entries of a matrix in the positions $\Omega$, 
probability of completing the matrix to a rank $r$ matrix with given known entries is 0.
\end{theorem}
\begin{proof}
We mainly use Lemma~\ref{fiberdimensionlemma}. Let $X=\overline{\mathcal{M}_r}$ which is an 
irreducible variety by Lemma~\ref{irreduciblity}. Let 
$Y=\Phi_\Omega(\overline{\mathcal{M}_r})$ which is also an irreducible variety as it is a 
continuous image of the irreducible variety $\overline{\mathcal{M}_r}$. 
 Clearly, $\Phi_\Omega$ is a regular map, 
we have $\dim\Phi_\Omega(\overline{\mathcal{M}_{r}}) \le 
\dim(\overline{\mathcal{M}_{r}})=2nr-r^2<m$. Thus,  
$\Phi_\Omega(\overline{\mathcal{M}_{r}})$ 
	is a proper lower dimensional closed subset in $\mathbb{C}^m$. 
For almost all points in $\mathbb{C}^m$, they do not belong to  
$\Phi_\Omega(\overline{\mathcal{M}_{r}})$. In other words, for almost all points  
${\bf x}\in \mathbb{C}^m$, there is no matrix  
$X\in \overline{\mathcal{M}_{r}}$ such that  $\Phi_\Omega(X)={\bf x}$.  
\end{proof}

Next define the subset $\chi_\Omega\subset \overline{\mathcal{M}_{r}}$ by
$$
\chi_\Omega = \left\{ X \in   \overline{\mathcal{M}_{r}}
\mid \Phi_\Omega^{-1}(\Phi_\Omega(X))\text{ is zero dimensional} \right\}. 
$$
As we are working over Noetherian fields like $\mathbb{R}$ or $\mathbb{C}$, it is worthwhile to keep in mind 
that all zero dimensional varieties over such fields will have only finitely many points. 
Next we recall the following result from  Proposition 11.12 in \cite{JoeHarris}.

\begin{lemma}\label{JoeHarrislemma}
Let $X$ be a quasi-projective variety and $\pi: X \rightarrow \mathbb{P}^m$ a regular map; let $Y$ be 
closure of the image. For any $p \in X$, let $X_p = \pi^{-1}\pi(p)) \subseteq X$ be the fiber of $\pi$	\
through $p$, and let $\mu(p) = \dim_p(X_p)$ be the local dimension of $X_p$ at $p$. Then $\mu(p)$ is an 
upper-semicontinuous function of $p$, in the Zariski topology on $X$ - that is, for any $m$ the locus of 
points $p \in X$ such that $\dim_p(X_p) > m$ is closed in $X$. Moreover, if $X_0 \subseteq X$ 
is any irreducible component, $Y_0 \subseteq Y$ the closure of its image and $\mu$ the minimum value 
	of $\mu(p)$ on $X_0$, then 
\begin{equation}
\label{dimension}
dim(X_0) = dim(Y_0) + \mu.
\end{equation}
\end{lemma}

As we saw that $\dim(\Phi_\Omega(\overline{\mathcal{M}_{r}})\le \dim(\overline{\mathcal{M}_{r}})$, 
we can be more precise about these dimensions as shown in the following 
\begin{lemma}
\label{dimensions_been_equal}
Assume $m  > \dim(\overline{\mathcal{M}_{r}})$. Then $\chi_\Omega$ is open subset of 
$\overline{\mathcal{M}_{r}}$ and $\dim(\overline{\mathcal{M}_{r}}) = \dim(\overline{\Phi_\Omega(\overline{\mathcal{M}_{r}})}) = \dim(\Phi_\Omega(\overline{\mathcal{M}_{r}})) $ if and only if $\chi_\Omega \neq \emptyset$.
\end{lemma}
\begin{proof} Assume $\dim(\overline{\mathcal{M}_{r}}) = 
\dim(\overline{\Phi_\Omega(\overline{\mathcal{M}_{r}})}) = \dim(\Phi_\Omega(\overline{\mathcal{M}_{r}}))$. Then 
using Lemma \ref{fiberdimensionlemma}, there exists a nonempty open subset $U \subset 
\Phi_\Omega(\overline{\mathcal{M}_{r}})$ such that $\dim(\Phi_\Omega^{-1}(y)) = 0$ for all $y \in U$. 
This implies that $\Phi_\Omega^{-1}(y) \in \chi_\Omega$. Hence $\chi_\Omega \neq \emptyset$.
	 	
We now prove the converse. Assume $\chi_\Omega \neq \emptyset$.
We will apply Lemma~\ref{JoeHarrislemma} above by setting $X = \overline{\mathcal{M}_{r_g}}$, 
$Y=\Phi_\Omega(\overline{\mathcal{M}_{r_g}})$ and $\pi = \Phi_\Omega$. 
Couple of things to note here are that it does not matter whether we take the closure in $\mathbb{P}^m$ or in $\mathbb{C}^m$ since $\mathbb{C}^m$ is an open set in $\mathbb{P}^m$ and the Zariski topology of the affine space $\mathbb{C}^m$ is induced from the Zariski topology of $\mathbb{P}^m$. $\overline{\mathcal{M}_{r_g}}$ is an affine variety. Therefore, it is a quasi-projective variety. 

By our assumption, $\chi_\Omega$ is not empty. It follows that there is a point $p \in Y$ 
such that $\pi^{-1}(p)$ is zero dimensional. Since zero is the least dimension possible, 
we have $\mu = 0$. Hence, using (\ref{dimension}) above, we have $\dim(\overline{\mathcal{M}_{r}}) = 
\dim(\overline{\Phi_\Omega(\overline{\mathcal{M}_{r}})})$. 
But dimension does not change upon taking closure. So, $\dim(\Phi_\Omega(\overline{\mathcal{M}_{r}})) = 
\dim(\overline{\Phi_\Omega(\overline{\mathcal{M}_{r}})})$. Also, using Lemma \ref{Mumford}, 
$\chi_\Omega = \{x \in X : \dim(\phi^{-1}\phi(x)) < 1\}$ is an open subset of $\overline{\mathcal{M}_r}$.
\end{proof}

In the proof above, the following result was used. See   I.8. Corollary 3 in \cite{Mumford}.
\begin{lemma}
\label{Mumford}
Let $\phi : X \rightarrow Y$ be a morphism of affine varieties. Let
$\phi^{-1}\phi(x) = Z_1\cup\cdots \cup Z_j$ be the irreducible components of $\phi^{-1}\phi(x)$. 
Let $e(x)$ be the maximum of the dimensions of the $Z_i, i=1, \cdots, j$. 
Let $S_n(\phi) := \{x \in X : e(x) \geq n\}$. Then, for any $n\ge 1$, $S_n(\phi)$ is a Zariski closed
subset of $X$. Equivalently $\{x \in X : \dim(\phi^{-1}\phi(x)) < n\}$ is an open subset of $X$.
\end{lemma}

Finally, we need the following  
\begin{definition}
	The \emph{degree} of an affine or projective variety of dimension $k$ is the number of intersection points of the variety with $k$ hyperplanes in general position.
\end{definition}

For example, the degree of the algebraic variety $\overline{\mathcal{M}_r}$ is known. 
See Example 14.4.11 in \cite{Fulton}, i.e.  
\begin{example}\label{degreelemma}
	Degree of the algebraic variety $\overline{\mathcal{M}_r}$ is 
	$$
	\prod_{i=0}^{n-r-1}\frac{\binom{n+i}{r}}{\binom{r+i}{r}}
	$$
\end{example}
We are now ready to prove another main result in this section. 
\begin{theorem}
\label{numberofcompletion}
Fix $\Omega$. Assume that there exist a finite $r$-feasible vector ${\bf x}\in \mathbb{C}^m$
over the given $\Omega$. Then, with probability 1, any r-feasible vector ${\bf y}$ is finitely $r$-feasible. 
In other words, if one randomly chooses a feasible vector ${\bf x}$ in the positions $\Omega$, then, with probability 1, 
the matrix can be completed into a rank-$r$ matrix only in \emph{finitely many ways}.
In  additional, the number of ways to complete will be less than or equal to $\displaystyle 
\prod_{i=0}^{n-r-1}\frac{\binom{n+i}{r}}{\binom{r+i}{r}}$.
\end{theorem}
\begin{proof} We begin by noting that, both $\overline{\mathcal{M}_r}$ 
and $\Phi_\Omega(\overline{\mathcal{M}_r})$ are irreducible varieties. So, the closure 
$\overline{\Phi_\Omega(\overline{\mathcal{M}_r})}$ is also an irreducible variety. 
By the assumption and using Lemma~\ref{dimensions_been_equal},  $\dim(\overline{\mathcal{M}_{r}}) = 
\dim(\overline{\Phi_\Omega(\overline{\mathcal{M}_{r}})})$. Hence, applying Lemma~\ref{fiberdimensionlemma}, 
there exist a nonempty open subset $U \subset \overline{\Phi_\Omega(\overline{\mathcal{M}_{r}})}$ such that 
$\Phi_\Omega^{-1}(y)$ is zero-dimensional for all $y \in U$.
In other words, If we choose the $m$ entries in positions $\Omega$ of a matrix from the open set $U$, then 
there are finitely many ways to complete the matrix. The result now follows by recalling that a Zariski open 
set in an irreducible variety is a dense set whose complement has Lebesque measure zero.

When we fix $m$ entries of a matrix $M$, the set of matrices of rank $r$ which has those entries in the positions $\Omega$ are exactly the intersection points of the variety $\overline{\mathcal{M}_r}$ with 
$m$ hyperplanes, namely the hyperplanes defined by equations of form $M_{ij} = constant$. Since 
$m>\dim(\overline{\mathcal{M}_r}) = 2nr-r^2$, the number of intersection points would be lesser than 
degree of $\overline{\mathcal{M}_r}$ generically. 
Now using the exact formula for the degree from Example~\ref{degreelemma}, the result follows.
\end{proof}

Regarding Theorem~\ref{fact8}, 
we have the following open problem: given ${\bf x}\in {\cal C}^m$, how to check if there are only finitely many
matrices $Y\in\overline{{\cal M}_r}$ satisfying $(Y)_\Omega= {\bf x}$. 


%


\noindent
{\bf Acknowledgment:} The first author would like to thank Professors Xiaofei He and 
Jieping Ye for their hospitality during his visits at Zheijiang University and 
University of Michigan in 2014. The second author would like to thank Anand Deopurkar for his helpful comments 
and suggestions in regard to the algebraic-geometric aspects of the paper.

\end{document}